\documentclass[1p,12pt]{elsarticle}

\usepackage{graphicx}
\usepackage{amsmath,amsxtra,amssymb,latexsym, amscd,amsthm}

\theoremstyle{plain}
\newtheorem{thm}{Theorem}[section]

\newtheorem{lem}[thm]{Lemma}

\theoremstyle{definition}

\newtheorem{rmk}[thm]{Remark}

\numberwithin{equation}{section}
\numberwithin{figure}{section}
\numberwithin{table}{section}
\newtheorem{conj}{Conjecture}
\usepackage{amsthm,amsmath,amssymb}

\usepackage[colorlinks=true,citecolor=blue,linkcolor=black,urlcolor=blue]{hyperref}

\usepackage{graphicx}

\newcommand{\TAD}{\operatorname{TAD}}

\newcommand{\CON}{\operatorname{CON}}
\newcommand{\CM}{\operatorname{CM}}
\newcommand{\AD}{\operatorname{AD}}
\newcommand{\AR}{\operatorname{AR}}
\newcommand{\W}{\operatorname{W}}
\newcommand{\F}{\operatorname{F}}
\newcommand{\Pol}{\operatorname{P}}
\newcommand{\SM}{\operatorname{SM}}

\begin{document}

\begin{frontmatter}
\title{\textbf{Proof of a conjecture of Kenyon and Wilson on semicontiguous minors}}


\author{Tri Lai\corref{cor1}\fnref{myfootnote1}}
\address{Department of Mathematics\\ University of Nebraska -- Lincoln\\ Lincoln, NE 68588}
\fntext[myfootnote1]{This research was supported in part by the Institute for Mathematics and its Applications with funds provided by the NSF (grant no. DMS-0931945).}
\cortext[cor1]{Corresponding author, email: tlai3@unl.edu, tel: 402-472-7001}

\begin{abstract}
Kenyon and Wilson  showed how to test if a circular planar electrical network  with $n$ nodes is well-connected by checking the positivity of $\binom{n}{2}$ central minors of the response matrix. Their test is based on the fact that any contiguous minor of a matrix can be expressed as a Laurent polynomial in the central minors.  Moreover, the Laurent polynomial is the generating function of domino tilings of a weighted Aztec diamond. They  conjectured that a larger family of minors, semicontiguous minors, can also be written in terms of domino tilings of a region on the square lattice. In this paper, we present a proof of the conjecture.
\end{abstract}


\end{frontmatter}

\section{Introduction}
The study of the electrical networks comes from classical physics with the work of Ohm and Kirchhoff  more than 100 years ago.
The \emph{circular planar electrical networks} were first studied systematically by Colin de Verdi\`{e}re \cite{Colin1} and Curtis, Ingerman, Moores, and Morrow \cite{Curtis1,Curtis2}.  Recently, a number of new properties of the circular planar electrical networks have been discovered (see e.g. \cite{Alman,KW1,KW,Lam,LP,Yi14}).

A \emph{circular planar electrical network} (or simply \emph{network} in this paper) is a finite graph $G=(V,E)$ embedded on a disk with a set of distinguished vertices $N\subseteq V$ on the circle, called \emph{nodes}, and a \emph{conductance function} $wt: E \rightarrow \mathbb{R}^+$ (see Figure \ref{circularnet} for an example).

\begin{figure}\centering
\includegraphics[width=6cm]{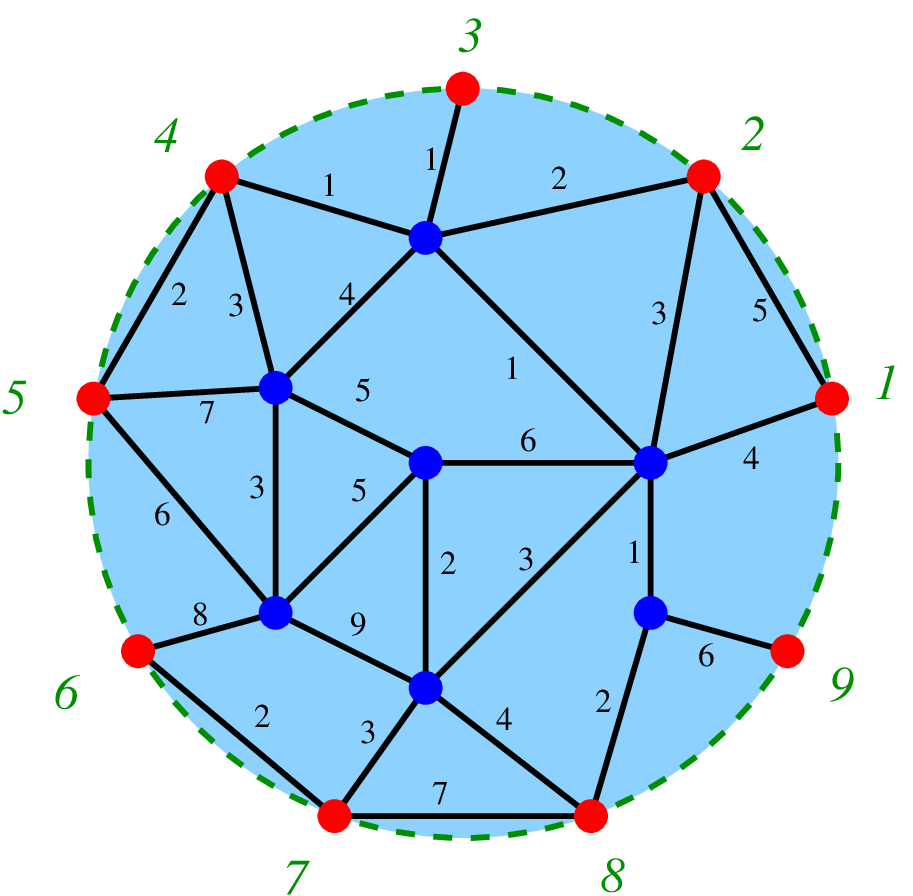}
\caption{A circular planar electrical network with 9 nodes.}
\label{circularnet}
\end{figure}

Arrange the indices $1,2,\dotsc, n$ of an $n\times n$ matrix $M=(m_{i,j})_{1\leq i,j\leq n}$ in counter-clockwise order around the circle.  Assume that  $A=\{a_1,a_2,\dotsc,a_k\}$ and $B=\{b_1,b_2,\dotsc,b_{\ell}\}$ are two sets of indices so that $a_1,a_2,\dotsc,a_k$ and $b_{\ell},b_{\ell-1},\dotsc,b_1$ are in counter-clockwise order around the circle. We denote by $M_{A}^{B}$ the submatrix $\big(m_{a_i,b_j}\big)_{\substack{\text{$1\leq i\leq k$}\\\text{$1\leq j \leq  \ell$}}}$ of $M$. In the case $k=\ell$, we call the pair $(A,B)$ a \emph{circular pair} of $M$ and the determinant $\det M_{A}^{B}$ a \emph{circular minor}\footnote{In this paper, we refer \emph{minors} as determinants of submatrices.} of $M$. If $A$ and $B$ are non-interlaced around the circle, we call the latter minor a \emph{non-interlaced circular minor}.

Associated with a network with $n$ nodes is a \emph{response matrix}  $\Lambda=(\lambda_{i,j})_{1\leq i,j\leq n}$ that measures the response of the network to potential applied at the nodes. In particular, $-\lambda_{i,j}$ is the current  that would flow into node $j$ if node $i$ is set to one volt and the remaining nodes are set to zero volts.
It has been shown that a matrix $M$ is the response matrix of a network if and only if it is symmetric with row and column sums equal to zero, and each non-interlaced circular minor $\det M_{A}^{B}$ is non-negative (see Theorem 4 in \cite{Curtis1}).

A network is called \emph{well-connected} if for any two non-interlaced sets of $k$ nodes $A$ and $B$, there are  $k$ pairwise vertex-disjoint paths in $G$ connecting the nodes in $A$ to the nodes in $B$.  A number of equivalent definitions of the well-connected networks were given in \cite{Colin1}. It has been shown by Colin de Verdi\`{e}re that a network is  well-connected if and only if all non-interlaced circular minors of the response matrix are positive.

A \emph{contiguous minor} of  a matrix $M$ is a circular minor of the form
\begin{equation}
\CON_{a,b,y}(M):=\det M_{A}^{B},
\end{equation}
where $A=\{a,a+1,\dots,a+y-1\}$ and $B=\{b+y-1,\dotsc,b+1,b\}$, and where the indices are interpreted modulo $n$ (i.e. the row indices and the column indices are contiguous on the circle).
 The \emph{central (contiguous) minor} $\CM_{x,y}(M)$ of $M$ is defined to be the contiguous minor $\CON_{a,b,y}(M)$, where the parameter $x$ satisfies the following conditions:
\begin{equation}\label{centralcondition}a=\left\lfloor\frac{x-y}{2}\right\rfloor \text{ and  } b=\left\lfloor\frac{x-y+n-(n-1 \mod 2)}{2}\right\rfloor\end{equation}
in modulo $n$.

 The central minor was first defined (implicitly)  in \cite{Curtis1}.
One readily sees that the parameter $x$ is naturally interpreted in modulo $2n$ (since increasing $x$ by 2 is equivalent to cyclically shifting the indices 1 unit counter-clockwise). The parameter $y$ ranges from $0$ to $n$.

If $1\leq x\leq n$, $1\leq y<n/2$ or $y=n/2$ and $x+y$ is odd, then we call $\CM_{x,y}(M)$ a \emph{small central minor}. It is easy to see that any small central minor is non-interlaced.

In this paper, we use the vertices of a regular $n$-gon arranged in the counter-clockwise order on a circle to represent the indices $1,2,\dots,n$ of a given $n\times n$ matrix $M$. A circular minor  $\det M_{A}^{B}$ of $M$ (where $A=\{a_1,\dots,a_k\}$ and $B=\{b_k,\dots,b_1\}$) is represented by  a circular diagram with $k$  chords connecting the vertex $a_i$ to the vertex $b_i$, for $i=1,2,\dots,k$. In this representation,  the central minors have their chords as centrally located as possible (plus or minus a rounding error).  Figure \ref{18gon} illustrates three circular minors of a $18\times 18$ matrix $M$: (a) a non-contiguous minor, (b) a contiguous minor, which is not a central minor, and (c) a central minor.
\begin{figure}\centering
\includegraphics[width=10cm]{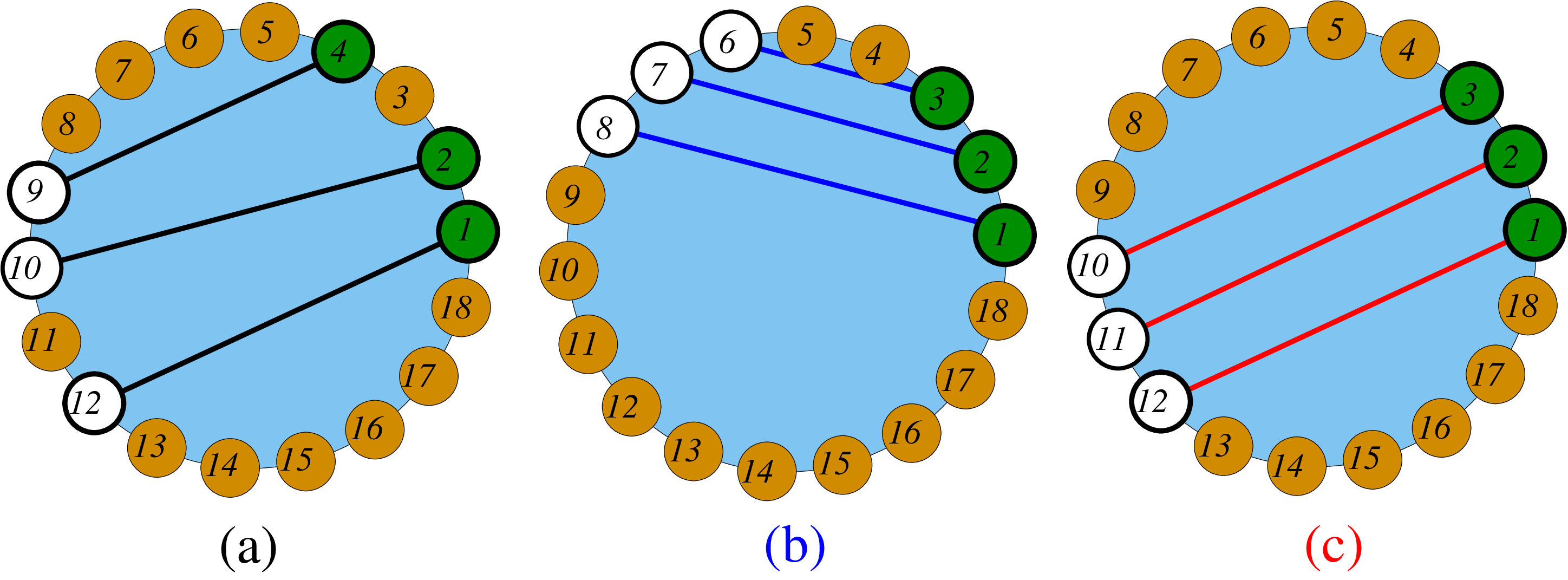}
\caption{Three circular minors of a $18\times 18$ matrix $M$: (a) $M_{1,2,4}^{12,10,9}$, (b) $\CON_{1,6,3}(M)$, and (c) $\CM_{6,3}(M)$.}
\label{18gon}
\end{figure}

There are total $\binom{n}{2}$ small central minors, whether $n$ is even or odd.  Kenyon and Wilson \cite{KW} showed how to test the well-connectivity of a network  by checking the positivity of the $\binom{n}{2}$ small central minors of the response matrix. This is a significant improvement of Colin de Verdi\`{e}re's previous test using exponentially many non-interlaced circular minors. Intuitively, one can say that the positivity of the central minors implies the positivity of all non-interlaced circular minors.

The \emph{Aztec diamond} $\AD_{x_0,y_0}^{h}$ of order $h$ with the center located at the lattice point $(x_0,y_0)$ in the grid $\mathbb{Z}^{2}$ is the region consisting of all unit squares inside the contour $|x-x_0|+|y-y_0|= h+1$.  It has been proven  that there are $2^{h(h+1)/2}$ different ways to cover an Aztec diamond of order $h$ by dominoes so that there are no gaps or overlaps; and such coverings are called \emph{domino tilings} of the Aztec diamond (see \cite{Elkies1,Elkies2}, and see \cite{Tri0,Tri2} for a recent generalization). The \emph{truncated Aztec diamond} $\TAD_{x_0,y_0}^{h,n}$ is defined to be the portion of the Aztec diamond $\AD_{x_0,y_0}^{h}$ between the two lines $y=0$  and $y=n$ (see Figure \ref{TAD} for several examples). We notice that $\TAD_{x_0,y_0}^{h,n}\equiv \AD_{x_0,y_0}^{h}$ if $h\leq y_0\leq n-h$.

\begin{figure}\centering
\includegraphics[width=13cm]{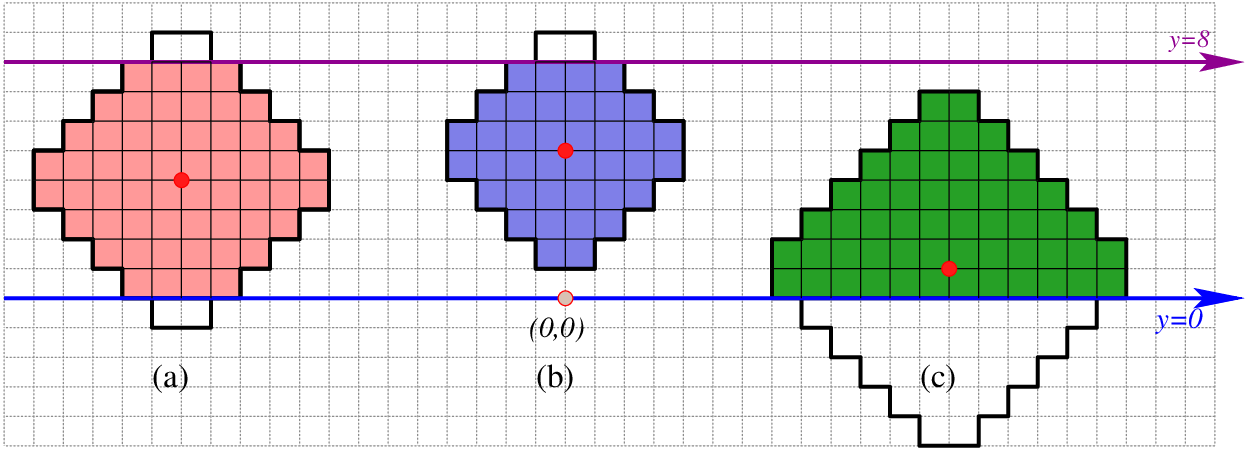}
\caption{The truncated Aztec diamonds:  (a) $\TAD_{-13,4}^{5,8}$, (b) $\TAD_{0,5}^{4,8}$, (c) $\TAD_{13,1}^{6,8}$.}
\label{TAD}
\end{figure}

Besides the Aztec diamonds, we are also interested in their natural generalizations, the \emph{Aztec rectangles}. The Aztec rectangle of size $3\times 6$ is illustrated in Figure \ref{ARminor}(a); the Aztec rectangle of size $4\times 6$ is shown in Figure \ref{ARminor}(b). The lattice point $(x_0,y_0)$ is called the \emph{center}  of the Aztec rectangle if the line $x=x_0$ passes through the middle point of the top length-2 step of the boundary, and the line $y=y_0$ passes through the middle point of the length-2 vertical step on the left of the boundary (see the dots in Figure \ref{ARminor}). Denote by $\AR_{x_0,y_0}^{m,n}$ the Aztec rectangle of size $m\times n$ with the center at $(x_0,y_0)$.

\begin{figure}\centering
\includegraphics[width=12cm]{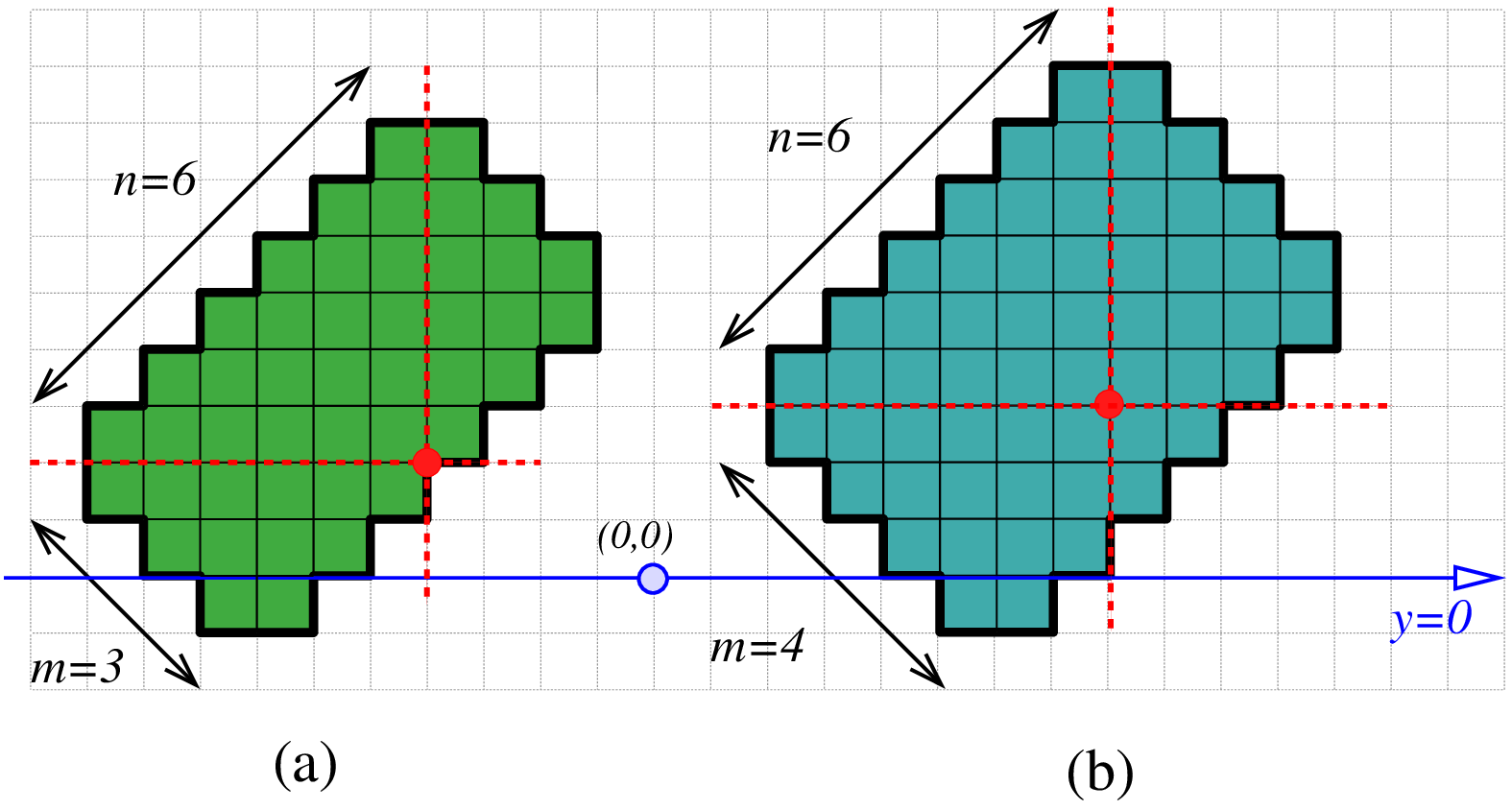}
\caption{The Aztec rectangles $\AR_{-4,2}^{3,6}$ (left) and $\AR^{4,6}_{8,3}$ (right).}
\label{ARminor}
\end{figure}

Given an $n\times n$ matrix $M$. For each lattice point $(x,y)$, we define  $v_{x,y}:=\CM_{x,y}$ if $0\leq y\leq n$, and $v_{x,y}:=1$ otherwise. We assign to each domino a weight $\frac{1}{v_{x_1,y_1}v_{x_2,y_2}}$, where $(x_1,y_1)$ and $(x_2,y_2)$ are the middle points of the long sides of the domino. In particular, the  horizontal domino consisting of the unit squares centered at $(x+\frac{1}{2},y+\frac{1}{2})$ and $(x-\frac{1}{2},y+\frac{1}{2})$ is weighted by $\frac{1}{v_{x,y}v_{x,y+1}}$; analogously, the vertical domino covering the unit squares centered at $(x+\frac{1}{2},y+\frac{1}{2})$ and $(x+\frac{1}{2},y-\frac{1}{2})$ has weight $\frac{1}{v_{x,y}v_{x+1,y}}$.
The \emph{weight} of a domino tiling of a \emph{region}\footnote{A \emph{region} considered in this paper is a finite connected union of unit squares of the square lattice}
 is the product of weights of all dominoes in the tiling. The weight $\W(R)$ of a  region $R$ is the sum of weights of all the domino tilings of $R$ (if $R$ does not have any domino tiling, then $\W(R):=0$; $\W(\emptyset):=1$ by convention). Our weight assignment here can be viewed as the `dual' of Speyer's weight assignment in \cite{Speyer}.

The \emph{covering monomial} $\F(R)$ of a non-empty region $R$ is defined to be the product $\prod_{x,y}v_{x,y}$ taken over all the lattice points $(x,y)$ inside $R$ or on the boundary of $R$, except for $90^{\circ}$-corners (these lattice points (x,y) are illustrated by the dots in Figure \ref{coveringmono}). Intuitively, the covering monomial of a region $R$ is the weight product of lattice points which the weight of dominoes in $R$ depends on.  The zero-order Aztec diamond $\AD_{x_0,y_0}^{0}$ is a formal empty region, which has the weight $\W(\AD_{x_0,y_0}^{0}):=1$ and the covering monomial $\F(\AD_{x_0,y_0}^{0}):=v_{x_0,y_0}$.

\begin{figure}\centering
\includegraphics[width=6cm]{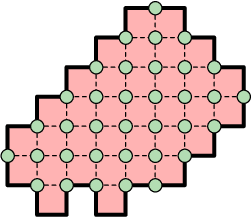}
\caption{Illustration of the definition of covering monomial.}
\label{coveringmono}
\end{figure}

To a region $R$, we associate a Laurent polynomial $\Pol(R):=\F(R)\W(R)$ in the variables $v_{x,y}$'s. We call $\Pol(R)$ the \emph{tiling polynomial} of the region $R$. Kenyon and Wilson \cite{KW} proved that any contiguous minor can be written as the tiling polynomial of a truncated Aztec diamond. The following theorem is a direct consequence of Kenyon and Wilson's Theorem 5 in \cite{KW}.

\begin{thm}\label{KWthm}
Let $\CON_{a,b,y}(M)$ be a contiguous minor of a matrix $M$. Assume that $h$ is the integer closest to $0$ so that $\CON_{a,b+h,y}(M)$ is the central minor $\CM_{x,y}(M)$. Then $\CON_{a,b,y}(M)=\Pol\left(\TAD_{x-h,y}^{|h|,n}\right)$, i.e. is the tiling polynomial of the truncated Aztec diamond $\TAD_{x-h,y}^{|h|,n}$.
\end{thm}
\begin{proof}
By Kenyon and Wilson's theorem (Theorem 5 in \cite{KW}), the Laurent polynomial $\Pol\left(\TAD_{x-h,y}^{|h|,n}\right)$ equals the contiguous minor $\det M_{A}^{B}$ whose row index set
$A$ is given by the row index set of the rightmost (or leftmost, respectively) central minor, and whose column
index set $B$ is given by the column index set of the leftmost (or rightmost, respectively) central minor. We only need to show that  $\CON_{a,b,y}(M)$ is exactly this contiguous minor.

Indeed, if $h$ is non-negative then the rightmost and the leftmost vertices, which are not a $90^{\circ}$ corner, in the truncated Aztec diamond $\TAD_{x-h,y}^{|h|,n}$ correspond to the central minors $\CM_{x,y}(M)$ and $\CM_{x-2h,y}(M)$, respectively. If $h$ is negative, these rightmost and the leftmost vertices correspond respectively to the central minors $\CM_{x-2h,y}(M)$ and $\CM_{x,y}(M)$. By the definition of $x,y,h$, it is easy to see that the row index set of $\CM_{x,y}(M)$ is the set $A=\{a,a+1,\dots,a+y-1\}$ and the column index set of $\CM_{x-2h,y}(M)$ is the set $B=\{b+y-1,\dots,b+1,b\}$. This means that $\CON_{a,b,y}(M)$ is the contiguous minor corresponding to the Laurent polynomial $\Pol\left(\TAD_{x-h,y}^{|h|,n}\right)$.
\end{proof}

For example, let $M$ be a $13\times 13$ matrix, then the contiguous minor $\CM_{1,5,2}(M)$ is expressed as the tiling polynomial  $\Pol(\TAD_{2,2}^{2,13})$ (shown in Figure \ref{diskexample}).

\begin{figure}\centering
\includegraphics[width=10cm]{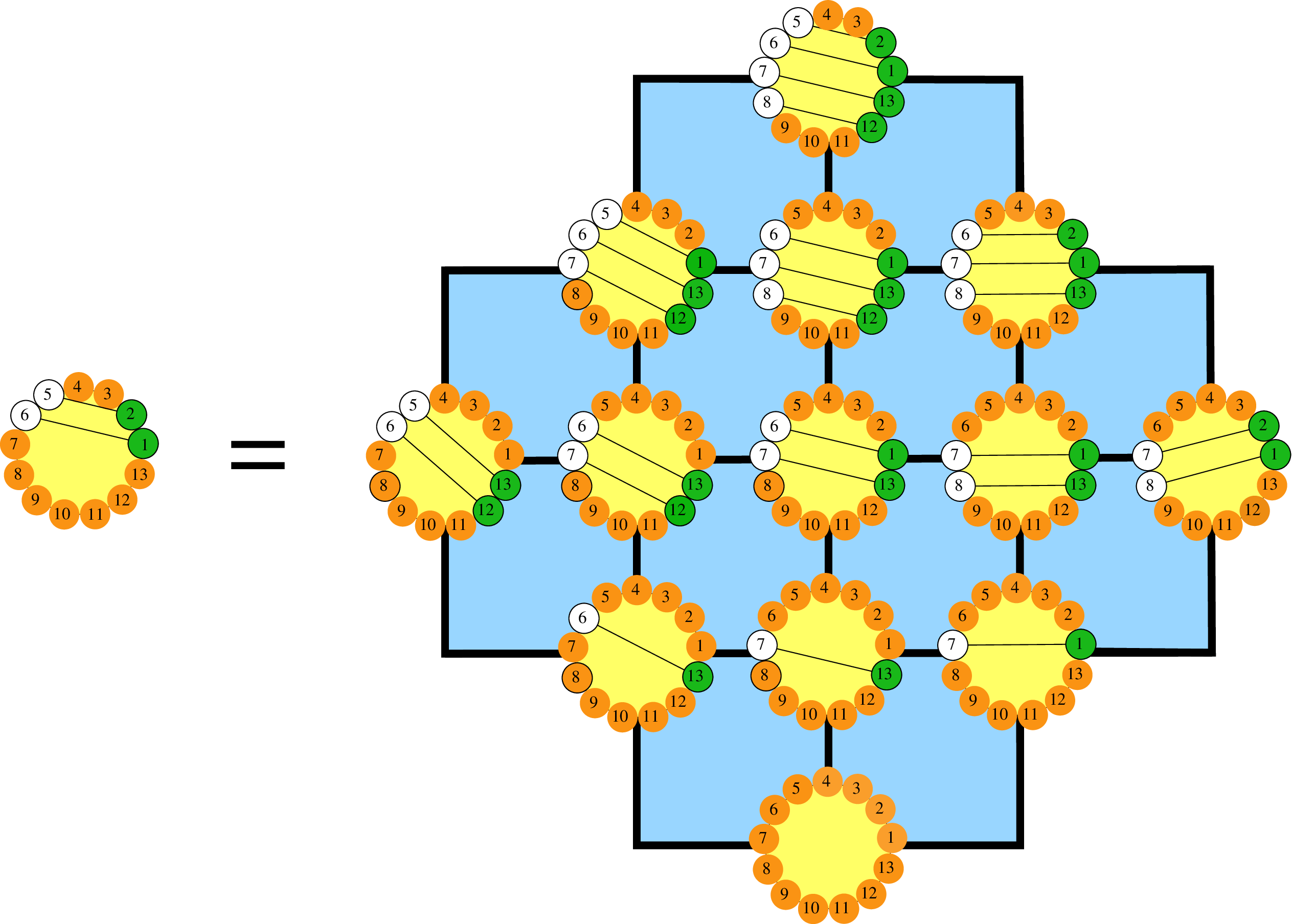}
\caption{The correspondence between contiguous minors and truncated Aztec diamonds.}
\label{diskexample}
\end{figure}

A \emph{semicontiguous minor} is a minor of the form $\det M_{A}^{B}$, where \emph{at least one} of $A$ and $B$ is contiguous\footnote{Strcitly speaking, this definition of the semicontiguous minors is slightly different from that in \cite{KW}. Kenyon and Wilson distinguished the contiguous minors from the semicontiguous minors by defining a  \emph{semicontiguous minor} to be the minor $\det M_{A}^{B}$, where \emph{exactly one} of $A$ and $B$ is contiguous. However, it is more convenient for us to view the contiguous minors as a special class of the semicontiguous minors as in our definition.}.
Kenyon and Wilson conjectured that
\begin{conj}[Conjecture 3 in \cite{KW}]\label{KWconj}
Any semicontiguous minor can be written as the tiling polynomial $\Pol(R)$ of some region $R$ on the square lattice.
\end{conj}
See Figure \ref{CM1} for an example. We refer the reader to \cite[pp. 25--27]{KW} for more examples.

\begin{figure}\centering
\includegraphics[width=13cm]{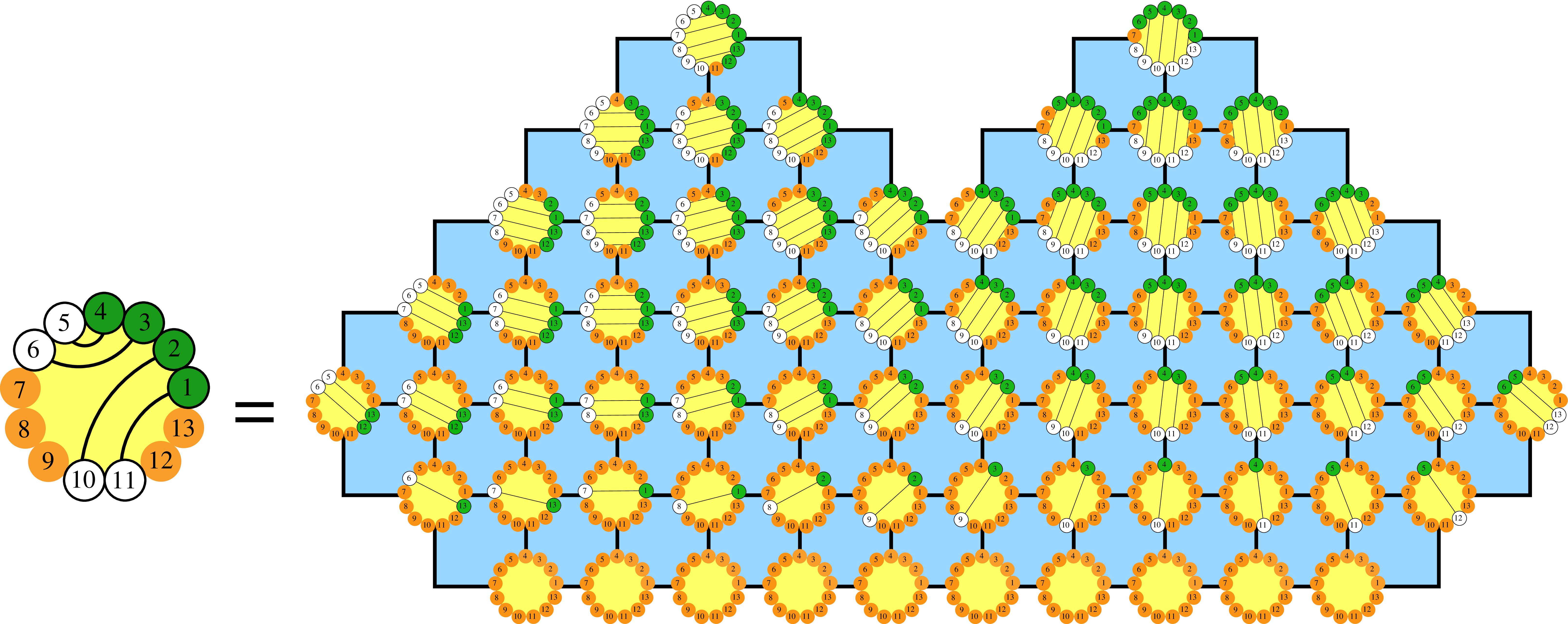}
\caption{Representing a semicontiguous minor as the tiling polynomial of a region on the square lattice.}\label{CM1}
\end{figure}

The goal of this paper is to prove this conjecture. Our proofs use a variation of \emph{Dodgson condensation} (or \emph{Desnanot-Jacobi identity}, see e.g.  \cite{Dodgson} and \cite[pp. 136--149]{Muir}) due to Kenyon and Wilson \cite{KW}  and a powerful method in the enumeration of tilings and perfect matchings, \emph{Kuo condensation} \cite{Kuo}. We refer the reader to e.g. \cite{Ciucu,YYZ,YZ,Ful,Speyer} for various aspects and generalizations of the method. More  recent applications of Kuo condensation can be found in e.g. \cite{CF14,CF15,CL,Tri1,Tri2,Tri3,Tri4,LMNT,KW}.

The rest of this paper is organized as follows. Our main results are presented in Section 2. In Section 3, we show the particular versions of Dodgson and Kuo condensations, which will be employed in our proofs.  The proofs of our main results will be shown in Section 4. Finally, we conclude the paper by posing an open question for the case of general circular minors.

\section{The main results}

In this section, we will describe carefully the structure of the regions corresponding to  the semicontiguous minors.

Consider a circular minor $\det M_{A}^{B}$ of an $n\times n$ matrix $M$, where at least one of $A$ and $B$ is contiguous. We consider first the case when $A$ is contiguous, then $B$ may be not contiguous. Assume that $B$ is partitioned into $s$ contiguous subsets $B_1,B_2,\dots,B_s$ (in counter-clockwise order around the circle). Assume in addition  that $|B_i|=k_i>0$, and that there are $t_i$ indices ($t_i>0$) separating $B_i$ and $B_{i+1}$. We call the sets of indices separating two consecutive subsets $B_i$'s the \emph{gaps} of $B$. Figure \ref{multigap}(a) shows an example of the semicontiguous minor with three gaps in $B$ for the case $n=60$, $s=4$, $k_1=3$, $k_2=4$, $k_3=3$, $k_4=2$, $t_1=t_2=t_3=2$ (the indices of each set $B_i$ are represented by adjacent nodes of the same color).   Denote by $k:=k_1+\dotsc+k_s$ and $t:=t_1+\dotsc+t_{s-1}$ (if $s=1$, then $t=0$ by convention). By definition, we always have $k+t\leq n$. We also assume that the first index in $A$ is $a$ (i.e., $A=\{a,a+1,\dotsc,a+k-1\}$), and the first index in $B$ is $b$. We use the notation\footnote{From now on, if the matrix $M$ is given, we usually drop the parameter $M$ in the notation of the $\SM$-minors.}
\[\SM_{a,b}(k_1,\dotsc,k_s;t_1,\dotsc,t_{s-1})=\SM_{a,b}(k_1,\dotsc,k_s;t_1,\dotsc,t_{s-1})(M)\]
for this minor. We notice that when $s=1$, $\SM_{a,b}(k_1,\dotsc,k_s;t_1,\dotsc,t_{s-1})$ is exactly the contiguous minor $\CON_{a,b,k_1}(M)$.

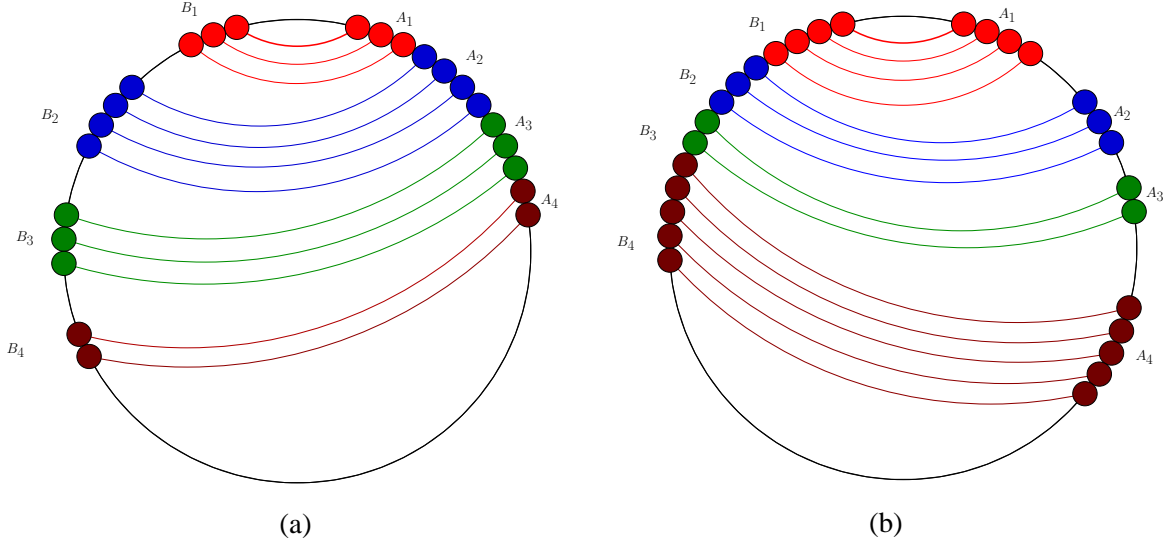
\begin{figure}\centering
\setlength{\unitlength}{3947sp}%
\begingroup\makeatletter\ifx\SetFigFont\undefined%
\gdef\SetFigFont#1#2#3#4#5{%
  \reset@font\fontsize{#1}{#2pt}%
  \fontfamily{#3}\fontseries{#4}\fontshape{#5}%
  \selectfont}%
\fi\endgroup%
\resizebox{15cm}{!}{\begin{picture}(0,0)%
\includegraphics{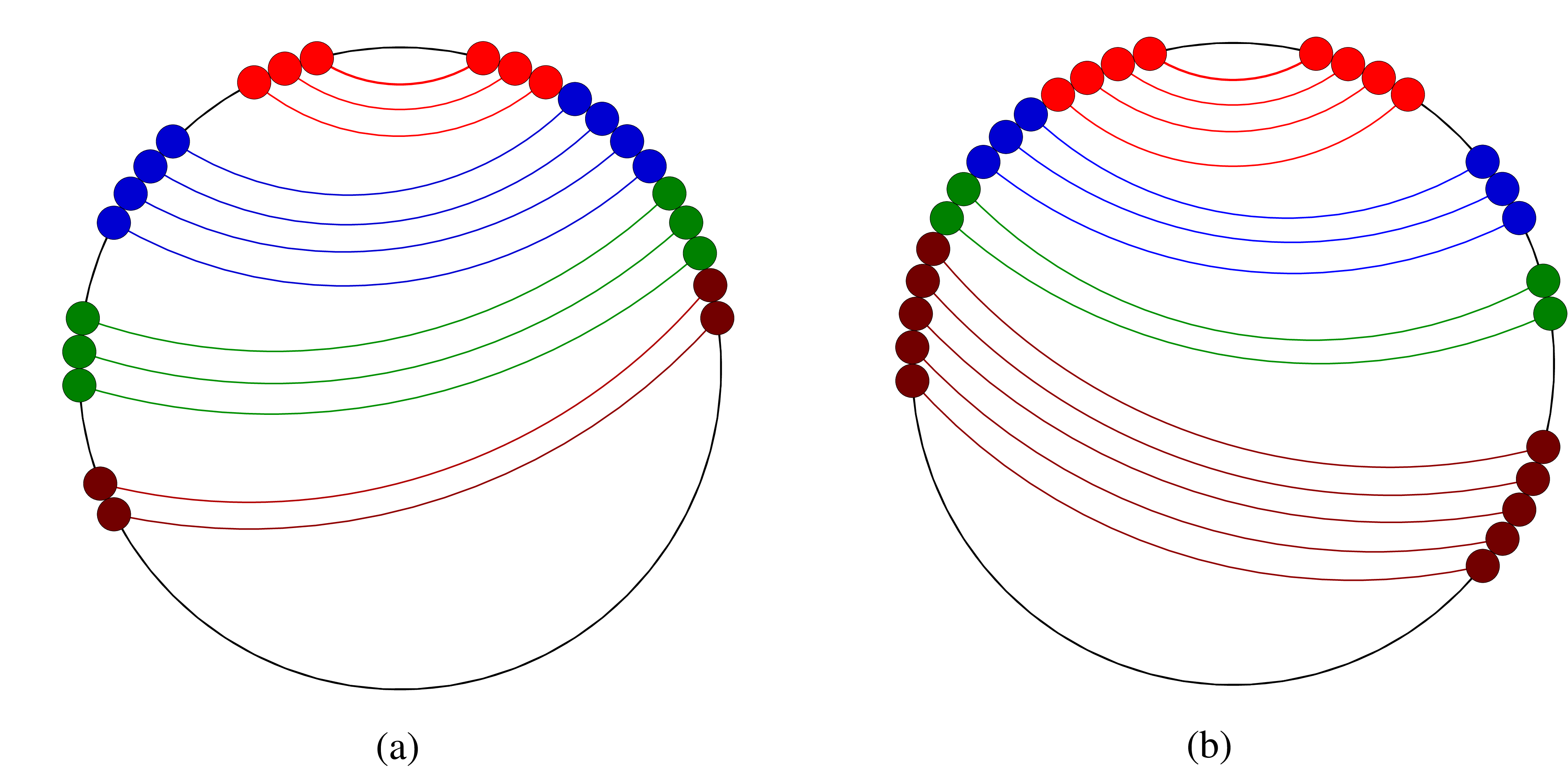}%
\end{picture}%
%
%

\begin{picture}(32793,16150)(646,-15492)
\put(11851,-391){\makebox(0,0)[lb]{\smash{{\SetFigFont{65}{65}{\rmdefault}{\mddefault}{\itdefault}{\color[rgb]{0,0,0}$A_1$}%
}}}}
\put(13861,-1501){\makebox(0,0)[lb]{\smash{{\SetFigFont{65}{49.2}{\rmdefault}{\mddefault}{\itdefault}{\color[rgb]{0,0,0}$A_2$}%
}}}}
\put(15271,-3421){\makebox(0,0)[lb]{\smash{{\SetFigFont{65}{49.2}{\rmdefault}{\mddefault}{\itdefault}{\color[rgb]{0,0,0}$A_3$}%
}}}}
\put(16021,-5686){\makebox(0,0)[lb]{\smash{{\SetFigFont{65}{49.2}{\rmdefault}{\mddefault}{\itdefault}{\color[rgb]{0,0,0}$A_4$}%
}}}}
\put(5670,-106){\makebox(0,0)[lb]{\smash{{\SetFigFont{65}{49.2}{\rmdefault}{\mddefault}{\itdefault}{\color[rgb]{0,0,0}$B_1$}%
}}}}
\put(931,-6751){\makebox(0,0)[lb]{\smash{{\SetFigFont{65}{49.2}{\rmdefault}{\mddefault}{\itdefault}{\color[rgb]{0,0,0}$B_3$}%
}}}}
\put(1651,-3211){\makebox(0,0)[lb]{\smash{{\SetFigFont{65}{49.2}{\rmdefault}{\mddefault}{\itdefault}{\color[rgb]{0,0,0}$B_2$}%
}}}}
\put(661,-10006){\makebox(0,0)[lb]{\smash{{\SetFigFont{65}{49.2}{\rmdefault}{\mddefault}{\itdefault}{\color[rgb]{0,0,0}$B_4$}%
}}}}
\put(21991,-541){\makebox(0,0)[lb]{\smash{{\SetFigFont{65}{49.2}{\rmdefault}{\mddefault}{\itdefault}{\color[rgb]{0,0,0}$B_1$}%
}}}}
\put(19981,-2071){\makebox(0,0)[lb]{\smash{{\SetFigFont{65}{49.2}{\rmdefault}{\mddefault}{\itdefault}{\color[rgb]{0,0,0}$B_2$}%
}}}}
\put(18841,-3691){\makebox(0,0)[lb]{\smash{{\SetFigFont{65}{49.2}{\rmdefault}{\mddefault}{\itdefault}{\color[rgb]{0,0,0}$B_3$}%
}}}}
\put(18211,-6901){\makebox(0,0)[lb]{\smash{{\SetFigFont{65}{49.2}{\rmdefault}{\mddefault}{\itdefault}{\color[rgb]{0,0,0}$B_4$}%
}}}}
\put(29172,-286){\makebox(0,0)[lb]{\smash{{\SetFigFont{65}{49.2}{\rmdefault}{\mddefault}{\itdefault}{\color[rgb]{0,0,0}$A_1$}%
}}}}
\put(32479,-3121){\makebox(0,0)[lb]{\smash{{\SetFigFont{65}{49.2}{\rmdefault}{\mddefault}{\itdefault}{\color[rgb]{0,0,0}$A_2$}%
}}}}
\put(33424,-5483){\makebox(0,0)[lb]{\smash{{\SetFigFont{65}{49.2}{\rmdefault}{\mddefault}{\itdefault}{\color[rgb]{0,0,0}$A_3$}%
}}}}
\put(33121,-10201){\makebox(0,0)[lb]{\smash{{\SetFigFont{65}{49.2}{\rmdefault}{\mddefault}{\itdefault}{\color[rgb]{0,0,0}$A_4$}%
}}}}
\end{picture}%
}
\caption{(a) The semicontiguous minor with gaps in $B$. (b) The  semicontiguous minor with gaps in $A$.}
\label{multigap}
\end{figure}

Let $A_1$ be the set consisting of the last $k_1$ indices in $A$, then $\det M_{A_1}^{B_1}$ is a circular minor of $M$.  We assume that $\TAD_{x-h,k_1}^{|h|,n}$ is the truncated Aztec diamond corresponding to the contiguous minor $\det M_{A_1}^{B_1}$ as in Theorem \ref{KWthm}, i.e.  $h$ is the integer closest to zero so that the contiguous minor $\det M_{A_1}^{B_1+h}$ is the central minor $\CM_{x,k_1}(M)$. Here $B_1+h$ is the set obtained from $B_1$ by translating it $|h|$ units counter-clockwise (resp., clockwise) if $h$ non-negative (resp., non-positive).  

If one fixes the first index $a$ of $A$, and let the first index $b$ of $B$  run along the circle, there are two cases in which $h=0$, i.e. when the contiguous minor $\det M_{A_1}^{B_1}$ is a central minor: $b=\lfloor \frac{n-1}{2}\rfloor+a+k-k_1$ and $b=\lfloor \frac{n-1}{2}\rfloor+a+k-k_1+1$.  We use the notations $0^+$ and $0^-$ to distinguish these zeros of $h$:  we define $h:=0^+$ if $b=\lfloor \frac{n-1}{2}\rfloor+a+k-k_1$, and $h:=0^-$ if $b=\lfloor \frac{n-1}{2}\rfloor+a+k-k_1+1$. We say that $h\geq 0^+$ if $h=0^+$ or $h\geq 1$, and that $h\leq 0^-$ if $h=0^-$ or $h\leq -1$.

\begin{figure}\centering

\setlength{\unitlength}{4144sp}%
\begingroup\makeatletter\ifx\SetFigFont\undefined%
\gdef\SetFigFont#1#2#3#4#5{%
  \reset@font\fontsize{#1}{#2pt}%
  \fontfamily{#3}\fontseries{#4}\fontshape{#5}%
  \selectfont}%
\fi\endgroup%

\resizebox{13cm}{!}{
\begin{picture}(0,0)%
\includegraphics{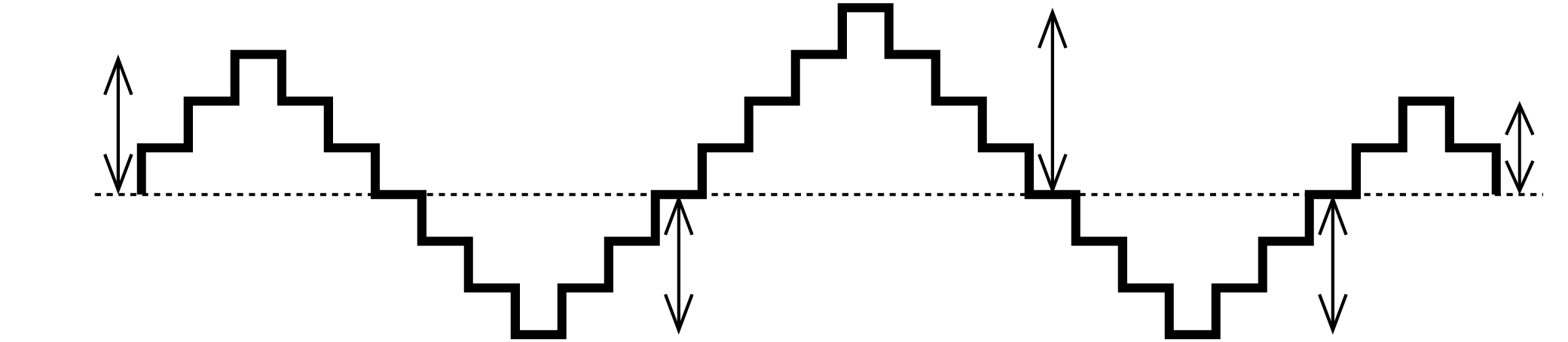}%
\end{picture}%

\begin{picture}(15857,3457)(-14,-4905)
\put(  1,-2940){\makebox(0,0)[lb]{\smash{{\SetFigFont{25}{30.0}{\rmdefault}{\mddefault}{\itdefault}{\color[rgb]{0,0,0}$k_1=3$}%
}}}}
\put(7088,-4358){\makebox(0,0)[lb]{\smash{{\SetFigFont{25}{30.0}{\rmdefault}{\mddefault}{\itdefault}{\color[rgb]{0,0,0}$t_1=3$}%
}}}}
\put(10867,-2468){\makebox(0,0)[lb]{\smash{{\SetFigFont{25}{30.0}{\rmdefault}{\mddefault}{\itdefault}{\color[rgb]{0,0,0}$k_2=4$}%
}}}}
\put(13938,-4358){\makebox(0,0)[lb]{\smash{{\SetFigFont{25}{30.0}{\rmdefault}{\mddefault}{\itdefault}{\color[rgb]{0,0,0}$t_2=3$}%
}}}}
\put(15828,-3177){\makebox(0,0)[lb]{\smash{{\SetFigFont{25}{30.0}{\rmdefault}{\mddefault}{\itdefault}{\color[rgb]{0,0,0}$k_3=2$}%
}}}}
\end{picture}}
\caption{The zigzag path $\mathcal{P}(3,4,2; 3,3)$.}
\label{zigzagbase}
\end{figure}
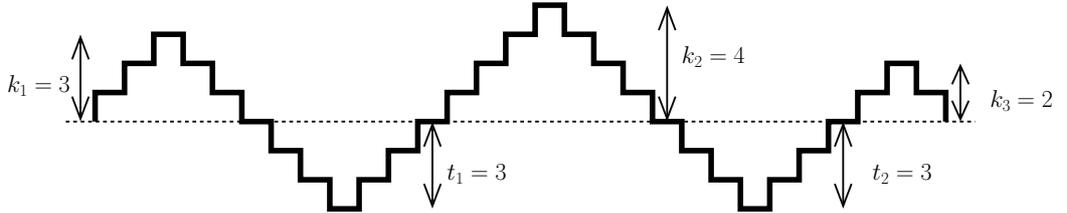

\medskip

Next, we encode the structure of the set $B$ by the zigzag path \\ $\mathcal{P}:=\mathcal{P}(k_1,\dotsc,k_s; t_1,\dotsc,t_{s-1})$ consisting of north and east steps, and starting and ending on the $x$-axis as follows.  $\mathcal{P}$  starts with a peak of height $k_1$, and contains alternatively a valley of depth $t_i$ and a peak of height $k_{i+1}$, for $i=1,2,\dotsc,s-1$ (see Figure \ref{zigzagbase} for an example). We use also the notation $\mathcal{P}^+:=\mathcal{P}^+(k_1,\dotsc,k_s; t_1,\dotsc,t_{s-1})$ (resp., $\mathcal{P}^-:=\mathcal{P}^-(k_1,\dotsc,k_s; t_1,\dotsc,t_{s-1})$) for the infinite lattice path obtained from $\mathcal{P}$ by extending horizontally to $+\infty$ (resp., $-\infty$) from the right (resp., left) endpoint.

\begin{figure}\centering
\includegraphics[width=12cm]{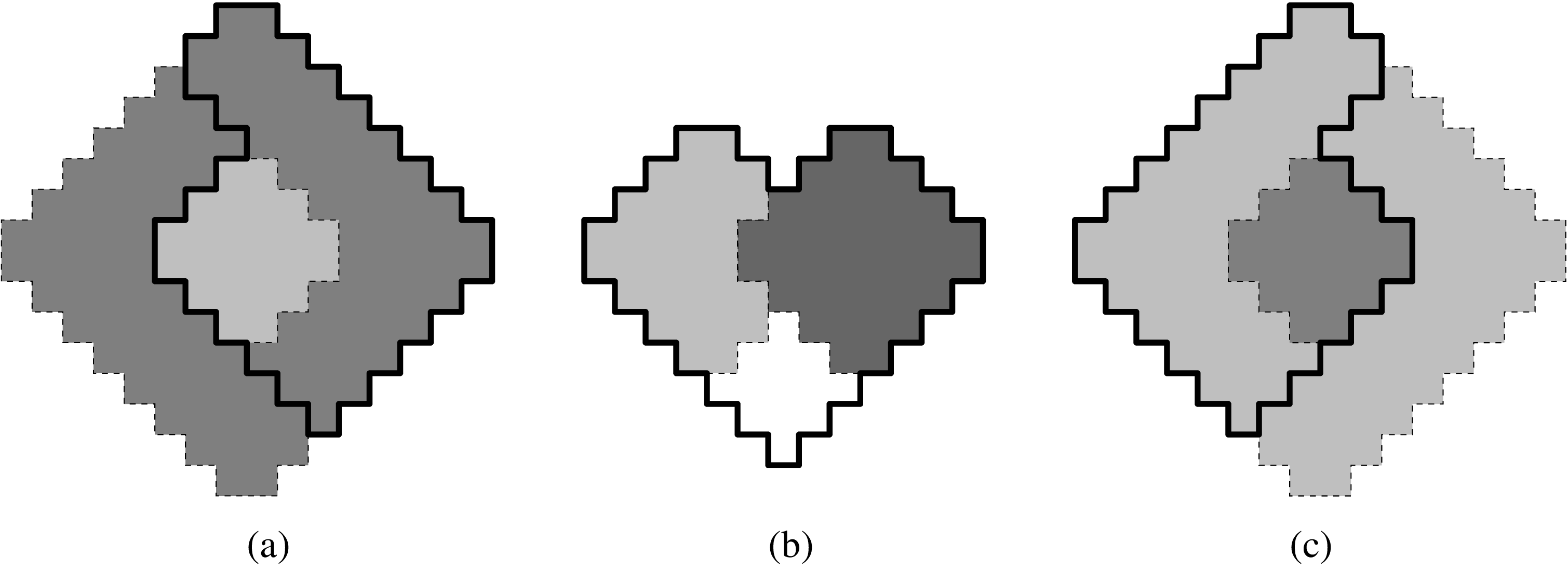}
\caption{The $L$-sum of two overlapping Aztec diamonds.}
\label{Lshape}
\end{figure}

\medskip

Before describing the family of regions corresponding to the $\SM$-minors, we define a family of new regions inspired by the Aztec diamond as follows. Consider any two overlapping Aztec diamonds $\AD_1:=\AD_{x_1,0}^{h_1}$ and $\AD_2:=\AD_{x_2,0}^{h_2}$ with center on the $x$-axis. We define the \emph{L-sum} $\AD_1\underset{L}{\oplus}\AD_2$ as in Figure \ref{Lshape}, where $\AD_1$ is the light shaded diamond and $\AD_2$ is the dark shaded one. More precisely, there are three cases to distinguish: if $\AD_1$ stays inside $\AD_2$, then $\AD_1\underset{L}{\oplus}\AD_2$ is the  $L$-shaped region restricted by the bold zigzag contour as in Figure \ref{Lshape}(a); if $\AD_2$ stays inside $\AD_1$, then $\AD_1\underset{L}{\oplus}\AD_2$ is the $L$-shaped region restricted by the bold zigzag contour as in Figure \ref{Lshape}(c); finally if the two diamonds do not contain each other, $\AD_1\underset{L}{\oplus}\AD_2$ is the $V$-shaped region as in Figure \ref{Lshape}(b). It is worth noticing that in the case when $AD_1$ and $AD_2$ do not contain each other, we have $\AD_1\underset{L}{\oplus}\AD_2 =\AD_2\underset{L}{\oplus}\AD_1$.

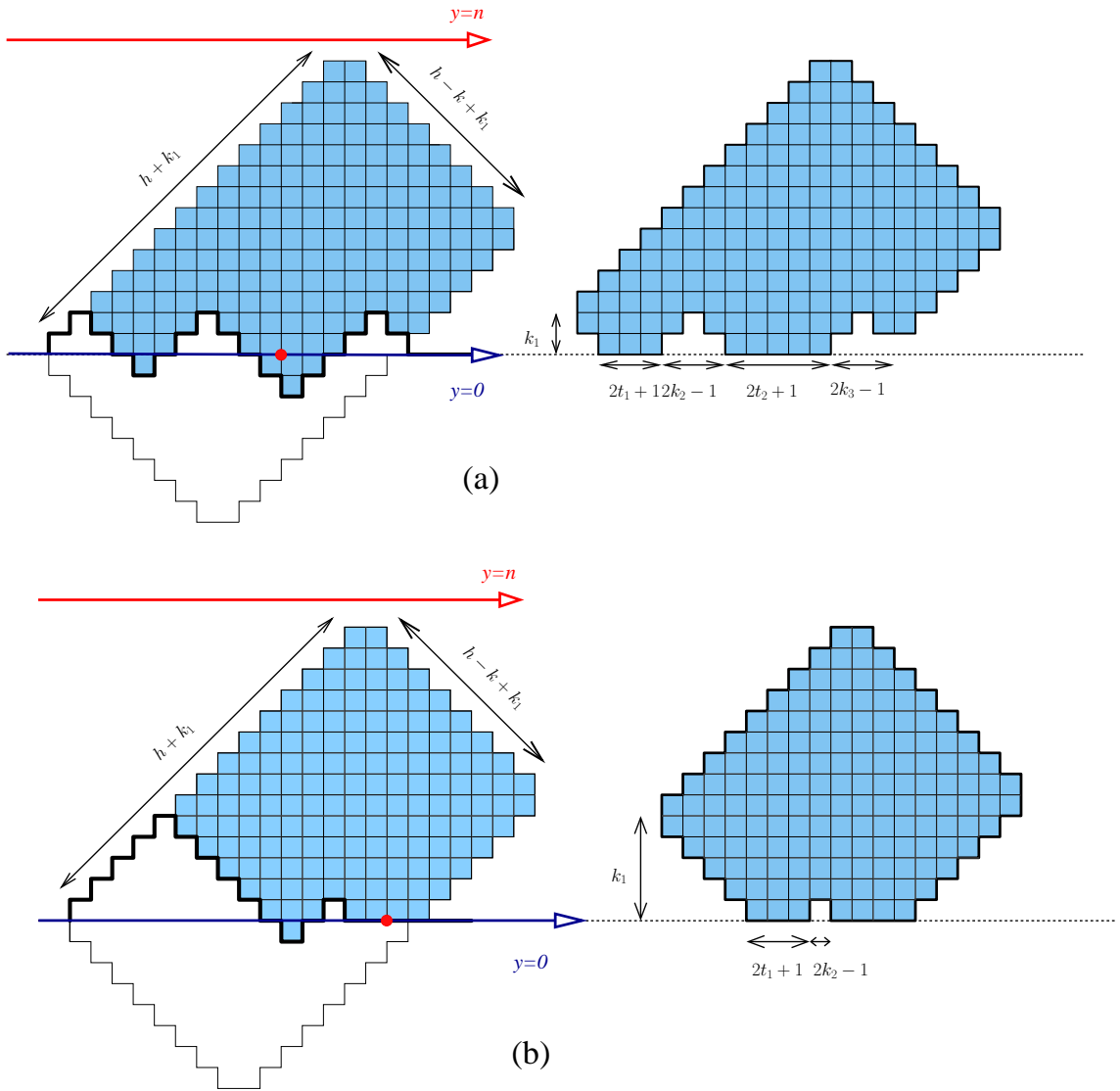
\begin{figure}\centering
\setlength{\unitlength}{3947sp}%
\begingroup\makeatletter\ifx\SetFigFont\undefined%
\gdef\SetFigFont#1#2#3#4#5{%
  \reset@font\fontsize{#1}{#2pt}%
  \fontfamily{#3}\fontseries{#4}\fontshape{#5}%
  \selectfont}%
\fi\endgroup%
\resizebox{15cm}{!}{
\begin{picture}(0,0)%
\includegraphics{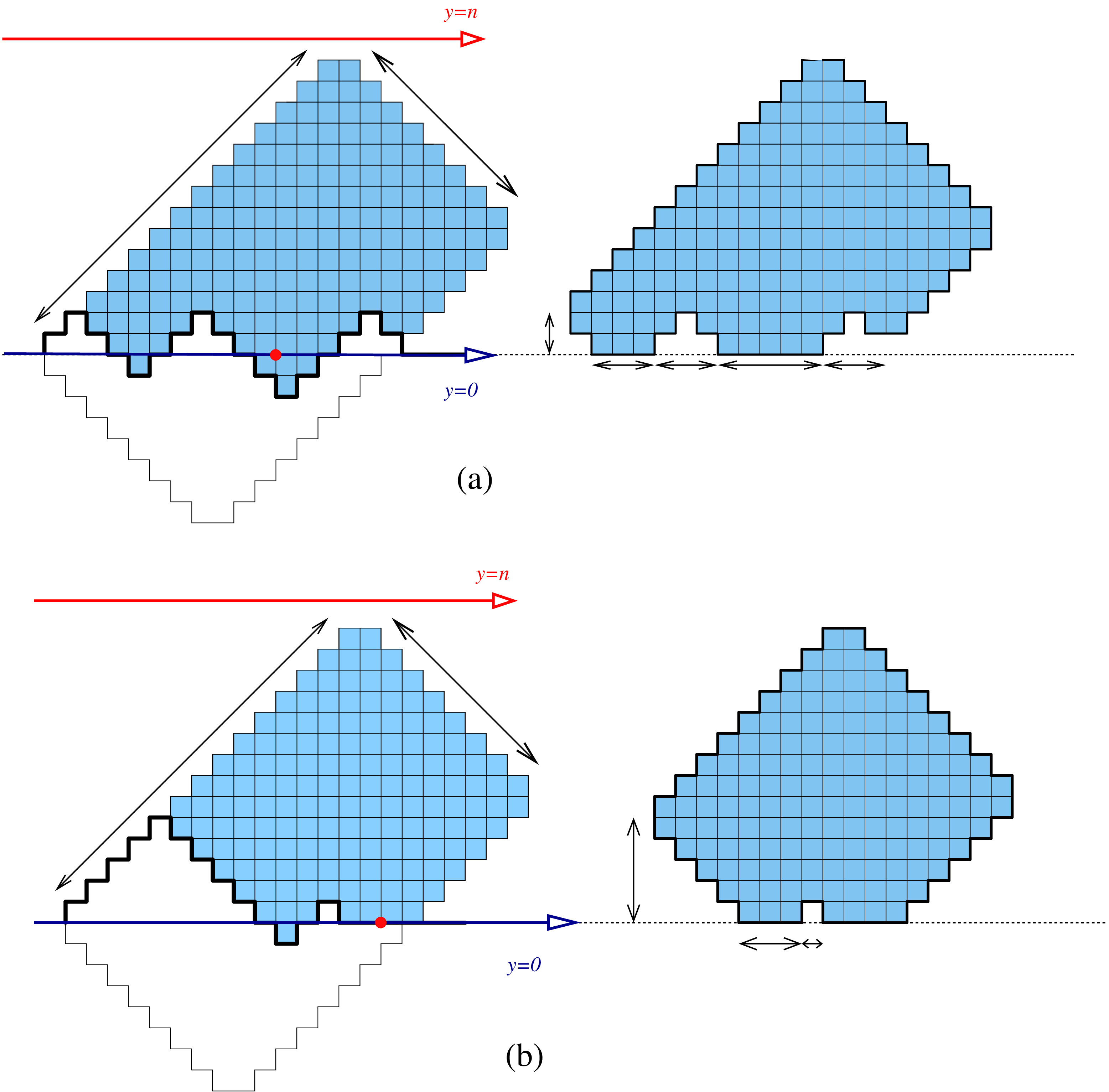}%
\end{picture}%
\begin{picture}(24891,24508)(1364,-25167)
\put(4539,-4830){\rotatebox{45.0}{\makebox(0,0)[lb]{\smash{{\SetFigFont{25}{30.0}{\rmdefault}{\mddefault}{\itdefault}{\color[rgb]{0,0,0}$h+k_1$}%
}}}}}
\put(10868,-2281){\rotatebox{315.0}{\makebox(0,0)[lb]{\smash{{\SetFigFont{25}{30.0}{\rmdefault}{\mddefault}{\itdefault}{\color[rgb]{0,0,0}$h-k+k_1$}%
}}}}}
\put(12992,-8374){\makebox(0,0)[lb]{\smash{{\SetFigFont{25}{30.0}{\rmdefault}{\mddefault}{\itdefault}{\color[rgb]{0,0,0}$k_1$}%
}}}}
\put(14883,-9555){\makebox(0,0)[lb]{\smash{{\SetFigFont{25}{30.0}{\rmdefault}{\mddefault}{\itdefault}{\color[rgb]{0,0,0}$2t_1+1$}%
}}}}
\put(16064,-9555){\makebox(0,0)[lb]{\smash{{\SetFigFont{25}{30.0}{\rmdefault}{\mddefault}{\itdefault}{\color[rgb]{0,0,0}$2k_2-1$}%
}}}}
\put(19831,-9526){\makebox(0,0)[lb]{\smash{{\SetFigFont{25}{30.0}{\rmdefault}{\mddefault}{\itdefault}{\color[rgb]{0,0,0}$2k_3-1$}%
}}}}
\put(14939,-20474){\makebox(0,0)[lb]{\smash{{\SetFigFont{25}{30.0}{\rmdefault}{\mddefault}{\itdefault}{\color[rgb]{0,0,0}$k_1$}%
}}}}
\put(18077,-22600){\makebox(0,0)[lb]{\smash{{\SetFigFont{25}{30.0}{\rmdefault}{\mddefault}{\itdefault}{\color[rgb]{0,0,0}$2t_1+1$}%
}}}}
\put(19450,-22600){\makebox(0,0)[lb]{\smash{{\SetFigFont{25}{30.0}{\rmdefault}{\mddefault}{\itdefault}{\color[rgb]{0,0,0}$2k_2-1$}%
}}}}
\put(4853,-17769){\rotatebox{45.0}{\makebox(0,0)[lb]{\smash{{\SetFigFont{25}{30.0}{\rmdefault}{\mddefault}{\itdefault}{\color[rgb]{0,0,0}$h+k_1$}%
}}}}}
\put(11605,-15329){\rotatebox{315.0}{\makebox(0,0)[lb]{\smash{{\SetFigFont{25}{30.0}{\rmdefault}{\mddefault}{\itdefault}{\color[rgb]{0,0,0}$h-k+k_1$}%
}}}}}
\put(17954,-9555){\makebox(0,0)[lb]{\smash{{\SetFigFont{25}{30.0}{\rmdefault}{\mddefault}{\itdefault}{\color[rgb]{0,0,0}$2t_2+1$}%
}}}}
\end{picture}}

\caption{Obtaining the type-1 region $\mathcal{Q}_{x,h}(k_1,\dots,k_s;t_1,\dots,t_{s-1})$ by truncating $\mathcal{H}_{x,h}(k_1,\dots,k_s;t_1,\dots,t_{s-1})$: (a) The case when $s=3$, $k_1=k_2=k_3=2$, $t_1=1$, $t_2=2$, $x=15$, $h=12$, (b) The case when $s=2$, $k_1=5$, $k_2=1$, $t_1=1$, $x=8$, $h=9$. The dots on the $x$-axis indicate the origin $(0,0)$.}
\label{newminor1}
\end{figure}

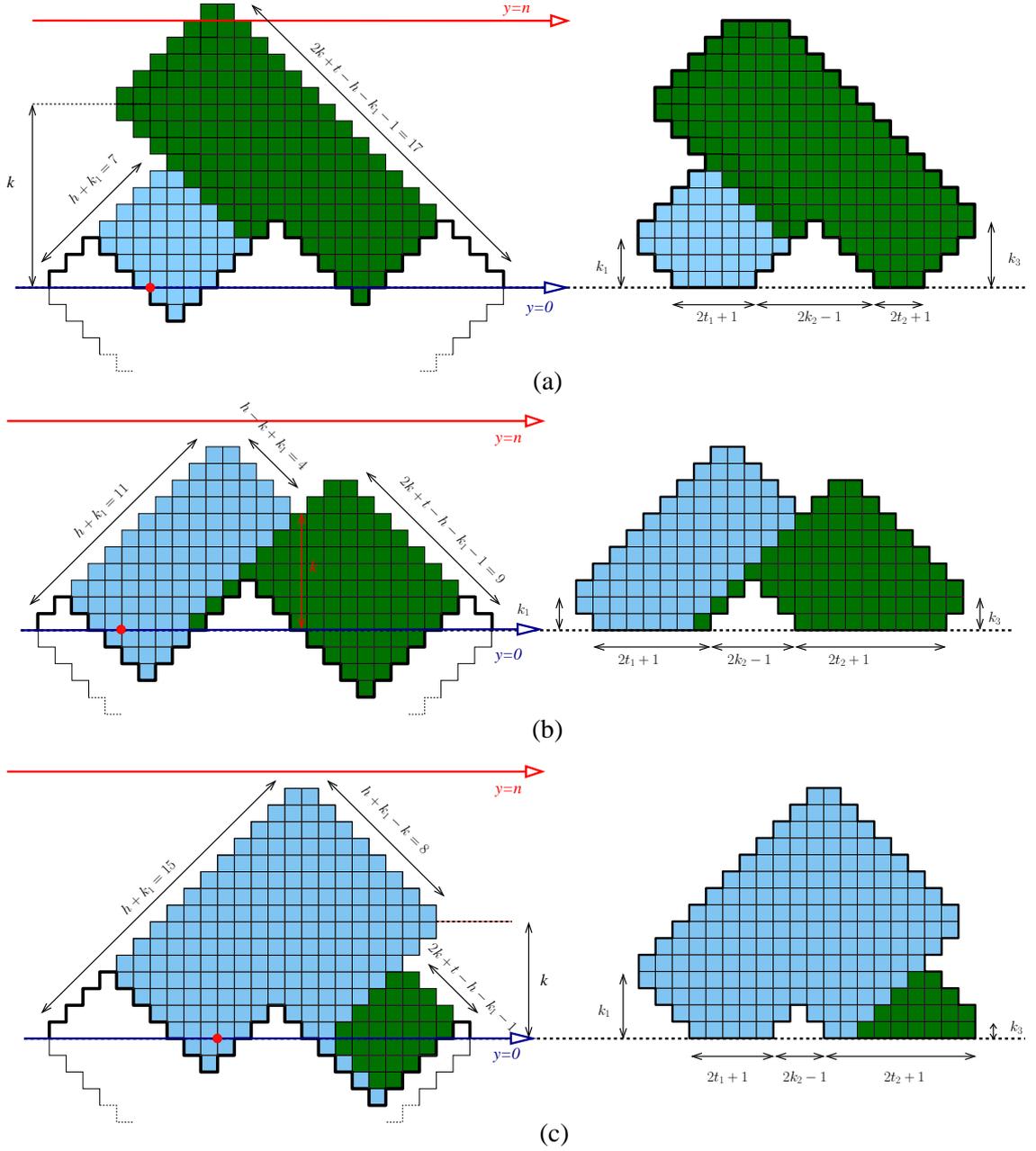
\begin{figure}\centering
\setlength{\unitlength}{3947sp}%
\begingroup\makeatletter\ifx\SetFigFont\undefined%
\gdef\SetFigFont#1#2#3#4#5{%
  \reset@font\fontsize{#1}{#2pt}%
  \fontfamily{#3}\fontseries{#4}\fontshape{#5}%
  \selectfont}%
\fi\endgroup%
\resizebox{15cm}{!}{
\begin{picture}(0,0)%
\includegraphics{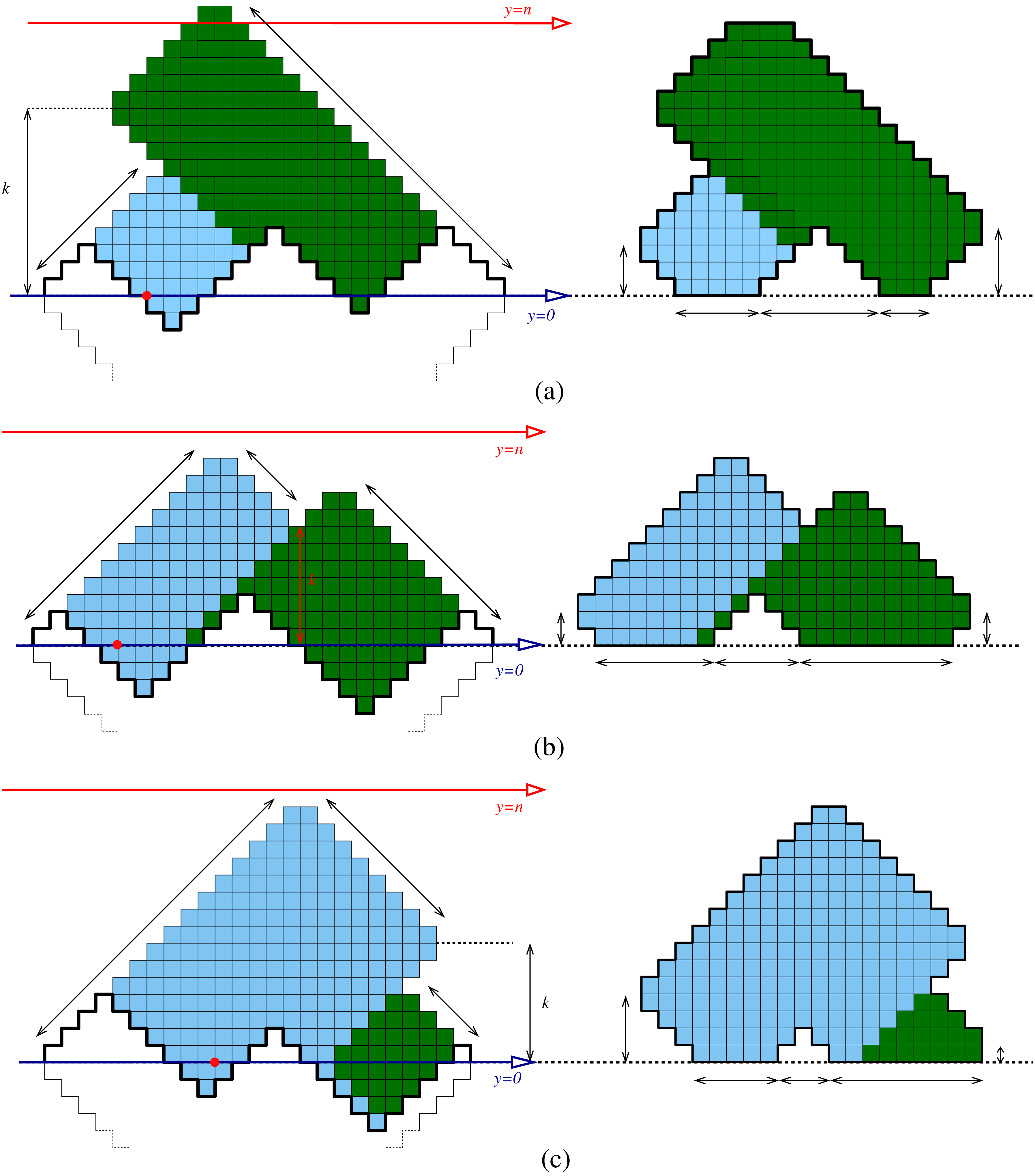}%
\end{picture}%
%
%

\begin{picture}(28694,32605)(2309,-32554)
\put(5726,-25908){\rotatebox{45.0}{\makebox(0,0)[lb]{\smash{{\SetFigFont{25}{30.0}{\rmdefault}{\mddefault}{\itdefault}{\color[rgb]{0,0,0}$h+k_1=15$}%
}}}}}
\put(12286,-22418){\rotatebox{315.0}{\makebox(0,0)[lb]{\smash{{\SetFigFont{25}{30.0}{\rmdefault}{\mddefault}{\itdefault}{\color[rgb]{0,0,0}$h+k_1-k=8$}%
}}}}}
\put(14174,-26860){\rotatebox{315.0}{\makebox(0,0)[lb]{\smash{{\SetFigFont{25}{30.0}{\rmdefault}{\mddefault}{\itdefault}{\color[rgb]{0,0,0}$2k+t-h-k_1-1$}%
}}}}}
\put(18952,-28725){\makebox(0,0)[lb]{\smash{{\SetFigFont{25}{30.0}{\rmdefault}{\mddefault}{\itdefault}{\color[rgb]{0,0,0}$k_1$}%
}}}}
\put(22023,-30614){\makebox(0,0)[lb]{\smash{{\SetFigFont{25}{30.0}{\rmdefault}{\mddefault}{\itdefault}{\color[rgb]{0,0,0}$2t_1+1$}%
}}}}
\put(24150,-30614){\makebox(0,0)[lb]{\smash{{\SetFigFont{25}{30.0}{\rmdefault}{\mddefault}{\itdefault}{\color[rgb]{0,0,0}$2k_2-1$}%
}}}}
\put(26985,-30614){\makebox(0,0)[lb]{\smash{{\SetFigFont{25}{30.0}{\rmdefault}{\mddefault}{\itdefault}{\color[rgb]{0,0,0}$2t_2+1$}%
}}}}
\put(30528,-29197){\makebox(0,0)[lb]{\smash{{\SetFigFont{25}{30.0}{\rmdefault}{\mddefault}{\itdefault}{\color[rgb]{0,0,0}$k_3$}%
}}}}
\put(10868,-1354){\rotatebox{315.0}{\makebox(0,0)[lb]{\smash{{\SetFigFont{25}{30.0}{\rmdefault}{\mddefault}{\itdefault}{\color[rgb]{0,0,0}$2k+t-h-k_1-1=17$}%
}}}}}
\put(18895,-7679){\makebox(0,0)[lb]{\smash{{\SetFigFont{25}{30.0}{\rmdefault}{\mddefault}{\itdefault}{\color[rgb]{0,0,0}$k_1$}%
}}}}
\put(21730,-9096){\makebox(0,0)[lb]{\smash{{\SetFigFont{25}{30.0}{\rmdefault}{\mddefault}{\itdefault}{\color[rgb]{0,0,0}$2t_1+1$}%
}}}}
\put(24564,-9096){\makebox(0,0)[lb]{\smash{{\SetFigFont{25}{30.0}{\rmdefault}{\mddefault}{\itdefault}{\color[rgb]{0,0,0}$2k_2-1$}%
}}}}
\put(27163,-9096){\makebox(0,0)[lb]{\smash{{\SetFigFont{25}{30.0}{\rmdefault}{\mddefault}{\itdefault}{\color[rgb]{0,0,0}$2t_2+1$}%
}}}}
\put(30470,-7443){\makebox(0,0)[lb]{\smash{{\SetFigFont{25}{30.0}{\rmdefault}{\mddefault}{\itdefault}{\color[rgb]{0,0,0}$k_3$}%
}}}}
\put(4253,-5781){\rotatebox{45.0}{\makebox(0,0)[lb]{\smash{{\SetFigFont{25}{30.0}{\rmdefault}{\mddefault}{\itdefault}{\color[rgb]{0,0,0}$h+k_1=7$}%
}}}}}
\put(4396,-15241){\rotatebox{45.0}{\makebox(0,0)[lb]{\smash{{\SetFigFont{25}{30.0}{\rmdefault}{\mddefault}{\itdefault}{\color[rgb]{0,0,0}$h+k_1=11$}%
}}}}}
\put(8892,-11521){\rotatebox{314.0}{\makebox(0,0)[lb]{\smash{{\SetFigFont{25}{30.0}{\rmdefault}{\mddefault}{\itdefault}{\color[rgb]{0,0,0}$h-k+k_1=4$}%
}}}}}
\put(13380,-13647){\rotatebox{314.0}{\makebox(0,0)[lb]{\smash{{\SetFigFont{25}{30.0}{\rmdefault}{\mddefault}{\itdefault}{\color[rgb]{0,0,0}$2k+t-h-k_1-1=9$}%
}}}}}
\put(16688,-17371){\makebox(0,0)[lb]{\smash{{\SetFigFont{25}{30.0}{\rmdefault}{\mddefault}{\itdefault}{\color[rgb]{0,0,0}$k_1$}%
}}}}
\put(19522,-18788){\makebox(0,0)[lb]{\smash{{\SetFigFont{25}{30.0}{\rmdefault}{\mddefault}{\itdefault}{\color[rgb]{0,0,0}$2t_1+1$}%
}}}}
\put(22593,-18788){\makebox(0,0)[lb]{\smash{{\SetFigFont{25}{30.0}{\rmdefault}{\mddefault}{\itdefault}{\color[rgb]{0,0,0}$2k_2-1$}%
}}}}
\put(25428,-18788){\makebox(0,0)[lb]{\smash{{\SetFigFont{25}{30.0}{\rmdefault}{\mddefault}{\itdefault}{\color[rgb]{0,0,0}$2t_2+1$}%
}}}}
\put(29916,-17607){\makebox(0,0)[lb]{\smash{{\SetFigFont{25}{30.0}{\rmdefault}{\mddefault}{\itdefault}{\color[rgb]{0,0,0}$k_3$}%
}}}}
\end{picture}
}
\caption{Obtaining the region $\mathcal{Q}_{x,h}(k_1,\dots,k_s;t_1,\dots,t_{s-1})$ of Type 2 from $\mathcal{H}_{x,h}(k_1,\dots,k_s;t_1,\dots,t_{s-1})$: (a) for $s=3$, $k_1=3$, $k_2=4$, $k_3=4$, $t_1=2$, $t_2=1$, $x=5$, $h=4$, (b) for  $s=3$, $k_1=2$, $k_2=3,$ $k_3=2$, $t_1=3$, $t_2=4$, $x=15$, $h=9$, and (c) for $s=3$, $k_1=4,$ $k_2=2$, $k_3=1$, $t_1=2$, $t_2=4$, $x=16$, $h=11$. The portion of $\AD_1$ has the light shading, and the portion of $\AD_2$ has the dark shading. }
\label{newminor3}
\end{figure}

\begin{figure}\centering
\setlength{\unitlength}{3947sp}%
\begingroup\makeatletter\ifx\SetFigFont\undefined%
\gdef\SetFigFont#1#2#3#4#5{%
  \reset@font\fontsize{#1}{#2pt}%
  \fontfamily{#3}\fontseries{#4}\fontshape{#5}%
  \selectfont}%
\fi\endgroup%
\resizebox{15cm}{!}{
\begin{picture}(0,0)%
\includegraphics{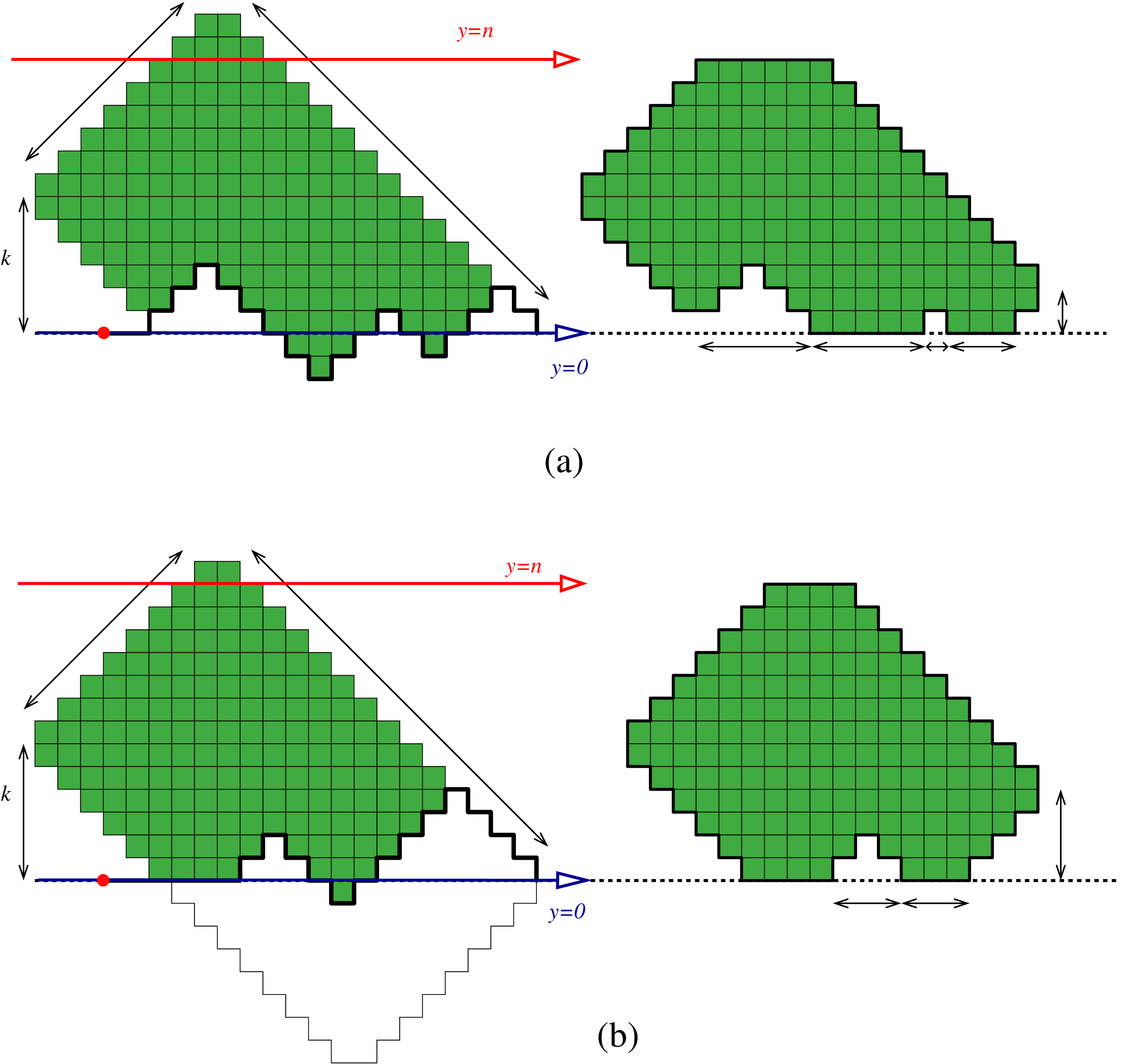}%
\end{picture}%
%
%

\begin{picture}(23219,22039)(695,-25167)
\put(1488,-5538){\rotatebox{45.0}{\makebox(0,0)[lb]{\smash{{\SetFigFont{25}{30.0}{\rmdefault}{\mddefault}{\itdefault}{\color[rgb]{0,0,0}$k+t-h-k_1=8$}%
}}}}}
\put(7805,-4358){\rotatebox{315.0}{\makebox(0,0)[lb]{\smash{{\SetFigFont{25}{30.0}{\rmdefault}{\mddefault}{\itdefault}{\color[rgb]{0,0,0}$2k+t-h-k_1=14$}%
}}}}}
\put(1561,-17056){\rotatebox{45.0}{\makebox(0,0)[lb]{\smash{{\SetFigFont{25}{30.0}{\rmdefault}{\mddefault}{\itdefault}{\color[rgb]{0,0,0}$k+t-h-k_1=8$}%
}}}}}
\put(7726,-15751){\rotatebox{315.0}{\makebox(0,0)[lb]{\smash{{\SetFigFont{25}{30.0}{\rmdefault}{\mddefault}{\itdefault}{\color[rgb]{0,0,0}$2k+t-h-k_1=14$}%
}}}}}
\put(22951,-20806){\makebox(0,0)[lb]{\smash{{\SetFigFont{25}{30.0}{\rmdefault}{\mddefault}{\itdefault}{\color[rgb]{0,0,0}$k_2$}%
}}}}
\put(19591,-22441){\makebox(0,0)[lb]{\smash{{\SetFigFont{25}{30.0}{\rmdefault}{\mddefault}{\itdefault}{\color[rgb]{0,0,0}$2t_1+1$}%
}}}}
\put(18061,-22501){\makebox(0,0)[lb]{\smash{{\SetFigFont{25}{30.0}{\rmdefault}{\mddefault}{\itdefault}{\color[rgb]{0,0,0}$2k_1-1$}%
}}}}
\put(15697,-10961){\makebox(0,0)[lb]{\smash{{\SetFigFont{25}{30.0}{\rmdefault}{\mddefault}{\itdefault}{\color[rgb]{0,0,0}$2k_1-1$}%
}}}}
\put(17947,-10931){\makebox(0,0)[lb]{\smash{{\SetFigFont{25}{30.0}{\rmdefault}{\mddefault}{\itdefault}{\color[rgb]{0,0,0}$2t_1+1$}%
}}}}
\put(19582,-10856){\makebox(0,0)[lb]{\smash{{\SetFigFont{25}{30.0}{\rmdefault}{\mddefault}{\itdefault}{\color[rgb]{0,0,0}$2k_2-1$}%
}}}}
\put(20887,-10826){\makebox(0,0)[lb]{\smash{{\SetFigFont{25}{30.0}{\rmdefault}{\mddefault}{\itdefault}{\color[rgb]{0,0,0}$2t_2+1$}%
}}}}
\put(23047,-9641){\makebox(0,0)[lb]{\smash{{\SetFigFont{25}{30.0}{\rmdefault}{\mddefault}{\itdefault}{\color[rgb]{0,0,0}$k_3$}%
}}}}
\end{picture}%
}
\caption{Obtaining the region $\mathcal{Q}_{x,h}(k_1,\dots,k_s;t_1,\dots,t_{s-1})$ of Type 3 from $\mathcal{H}_{x,h}(k_1,\dots,k_s;t_1,\dots,t_{s-1})$: (a) The example for $s=3$, $k_1=3$, $k_2=1$, $k_3=2$, $t_1=2$, $t_2=1$, $x=0$, $h=-2$. (b) The example for $s=2$, $k_1=2$, $k_2=4$, $t_1=1$, $x=1$, $h=-3$.}
\label{newminor2}
\end{figure}

\bigskip

We are now ready to define the family of regions \[\mathcal{Q}=\mathcal{Q}_{x,h}(k_1,\dots,k_s;t_1,\dots,t_{s-1})\] corresponding to the  minor  $\SM_{a,b}(k_1,\dotsc,k_s;t_1,\dotsc,t_{s-1})$ as follows.

\medskip

 There are three types of the regions $\mathcal{Q}$ as follows:

\begin{enumerate}
\item[\emph{Type 1.}] $t\leq h-k$.

  \quad Remove all unit squares in the Aztec rectangle $\AR_{x-h,0}^{h+k_1,h-k+k_1}$, which are below the zigzag path $\mathcal{P}^+:=\mathcal{P}^+(k_1,\dotsc,k_s; t_1,\dotsc,t_{s-1})$ with the left endpoint at the left corner of the Aztec rectangle (see the bold zigzag paths on the left pictures in Figure \ref{newminor1}). We get the region $\mathcal{H}=\mathcal{H}_{x,h}(k_1,\dots,k_s;t_1,\dots,t_{s-1})$ (shown by the shaded region on the left pictures in Figure \ref{newminor1}). Finally,  we define  $\mathcal{Q}$ to be the portion of  $\mathcal{H}$ between the lines $y=0$ and $y=n$.

\item[\emph{Type 2.}] $h\geq 0^+$ and $t> h-k$.

 \quad The region $\mathcal{Q}$ is obtained by applying the above double-trimming process to the region $\mathcal{R}:=\AD_1\underset{L}{\oplus}\AD_2$, where $\AD_1:=\AD_{x-h,0}^{h+k_1}$ and $\AD_2:=\AD_{x-h+t,0}^{2k+t-h-k_1-1}$ (instead of the Aztec rectangle $\AR_{x-h,0}^{h+k_1,h-k+k_1}$ as in type 1). In particular, $\mathcal{H}$ is the portion of $\mathcal{R}$ above the zigzag path $\mathcal{P}^+$; and our region $\mathcal{Q}$ is obtained from $\mathcal{H}$ by truncating the part below the line $y=0$ and the part above the line $y=n$. See Figure \ref{newminor3} for three examples corresponding to three possible shapes of the $L$-sum $\mathcal{R}$ as described in Figure \ref{Lshape}.

\item[\emph{Type 3.}] $h\leq 0^-$.

 \quad We start with the Aztec rectangle $\AR_{x-h+t,k}^{k+t-h-k_1,2k+t-h-k_1}$, then remove all unit squares below the zigzag path $\mathcal{P}^-$ with the \emph{right} endpoint at the \emph{right} corner of the rectangle, and truncate the part below the line $y=0$ and the part above the line $y=n$ from the resulting region. Two examples of the region $\mathcal{Q}$ in this case are shown in Figure \ref{newminor2}.
\end{enumerate}

\begin{rmk}\label{heightrmk}\begin{enumerate}
\item[(1).] If $s=1$, then $\mathcal{Q}=\mathcal{Q}_{x,h}(k_1;\emptyset):=\TAD_{x-h,k_1}^{|h|,n}$, the truncated Aztec diamond corresponding to the contiguous minor $\det M_{A_1}^{B_1}$ as in Theorem \ref{KWthm}.

\item[(2).] We notice that in Type 1, the top of the region $\mathcal{H}$ is always below or on the line $y=n$ (as $h+k_1\leq h+k+t\leq 2h\leq n$, since we are assuming that $t\leq h-k$, and $|h|\leq n/2$ by definition). However, in Types 2 and 3, the top of the region $\mathcal{H}$ may stay above the line $y=n$.
\end{enumerate}
\end{rmk}

\begin{thm}\label{genthm}
Assume that $s,k_1,\dotsc,k_s,t_1,\dots,t_{s-1}$ are positive integers, and that $M$ is an $n\times n$ matrix. Assume in addition that $\TAD_{x-h,k_1}^{|h|,n}$ is the truncated Aztec diamond corresponding to the contiguous minor $\det M_{A_1}^{B_1}$ defined as in Theorem \ref{KWthm}. Then
\begin{equation}\label{geneq1}
\SM_{a,b}(k_1,\dotsc,k_s;t_1,\dotsc,t_{s-1})=\Pol(\mathcal{Q}_{x,h}(k_1,\dots,k_s;t_1,\dots,t_{s-1})).
\end{equation}
\end{thm}

Next, we describe the region corresponding to the semicontiguous minor $\det M_{A}^{B}$, where $B$ is contiguous (and $A$ may contain some gaps).

 We now assume the decomposition $A=\bigcup_{i=1}^{s}A_i$, where $A_i$'s are contiguous index sets appearing in  \emph{clockwise} order around the circle. Assume in addition that $|A_i|=k_i>0$, and that the size of the gap between $A_i$ and $A_{i+1}$ is $t_i>0$ (see Figure \ref{multigap}(b) for an example for $n=60$, $s=4$, $k_1=4$, $k_2=3$, $k_3=2$, $k_4=5$, $t_1=2$, $t_2=1$, $t_3=3$). We also assume that $a$ and $b$ are respectively the first indices of $A$ and $B$ as usual (i.e., we have now $B=\{b,b+1,\dotsc,b+k-1\}$). Denote by \[\overline{\SM}_{a,b}(k_1,\dotsc,k_s;t_1,\dotsc,t_{s-1})=\overline{\SM}_{a,b}(k_1,\dotsc,k_s;t_1,\dotsc,t_{s-1})(M)\] this minor. We also note that $\overline{\SM}_{a,b}(k_1,\dotsc,k_s;t_1,\dotsc,t_{s-1})$ is contiguous when $s=1$.

 \medskip

Intuitively, the region $\overline{\mathcal{Q}}_{x,h}(k_1,\dots,k_s;t_1,\dots,t_{s-1})$ corresponding to the above  $\overline{SM}$-minor is obtained from the region $\mathcal{Q}_{x,h}(k_1,\dots,k_s;$ $t_1,\dots,t_{s-1})$ by reflecting it over a vertical line and translating horizontally. In particular,  the region \\ $\overline{\mathcal{Q}}_{x,h}(k_1,\dots,k_s;t_1,\dots,t_{s-1})$ is obtained by truncating the portion below the line $y=0$ and truncating the portion above the line $y=n$ of the region  \[\overline{\mathcal{H}}=\overline{\mathcal{H}}_{x,h}(k_1,\dots,k_s;t_1,\dots,t_{s-1})\] that is defined as follows. If $t\leq h-k$, $\overline{\mathcal{H}}$ is obtained from the  Aztec rectangle $\AR_{x-h,k}^{h-k+k_1,h+k_1}$ by removing all unit squares below the zigzag path $\overline{\mathcal{P}}^-$, where $\overline{\mathcal{P}}:=\mathcal{P}(k_s,\dotsc,k_1; t_{s-1},\dotsc,t_1)$, with the right endpoint at the rightcorner of the Aztec rectangle (see the left picture on the top row in Figure \ref{Newreverse}); if $h\geq 0^+$ and $t> h-k$, then $\overline{\mathcal{H}}$ is obtained by the same process for the $L$-sum $\overline{\mathcal{R}}:=\AD_{x-h,0}^{h+k_1}\underset{L}{\oplus}\AD_{x-h-t,0}^{2k+t-h-k_1-1}$  (illustrated by the left picture on the middle row in Figure \ref{Newreverse}); finally if $h\leq 0^-$, then $\overline{\mathcal{H}}$ is obtained from $\AR_{x-h-t,0}^{k+t-k_1-h,k+t-h}$ by removing all the unit squares below the zigzag line $\overline{\mathcal{P}}^+$ with the left endpoint at the left corner of the Aztec rectangle  (pictured by the left picture on the bottom row in Figure \ref{Newreverse}). 

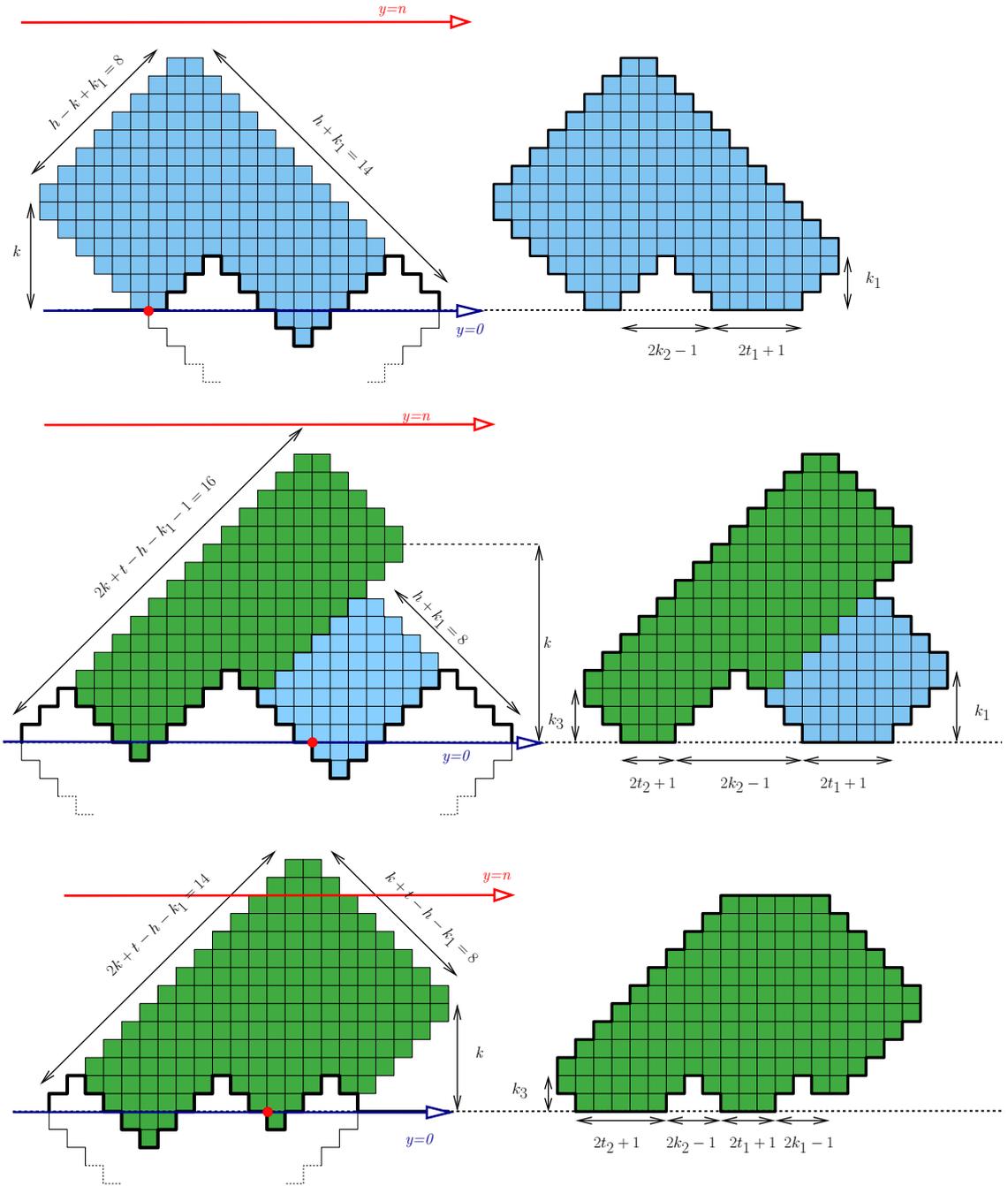
\begin{figure}\centering
\setlength{\unitlength}{3947sp}%
\begingroup\makeatletter\ifx\SetFigFont\undefined%
\gdef\SetFigFont#1#2#3#4#5{%
  \reset@font\fontsize{#1}{#2pt}%
  \fontfamily{#3}\fontseries{#4}\fontshape{#5}%
  \selectfont}%
\fi\endgroup%
\resizebox{15cm}{!}{
\begin{picture}(0,0)%
\includegraphics{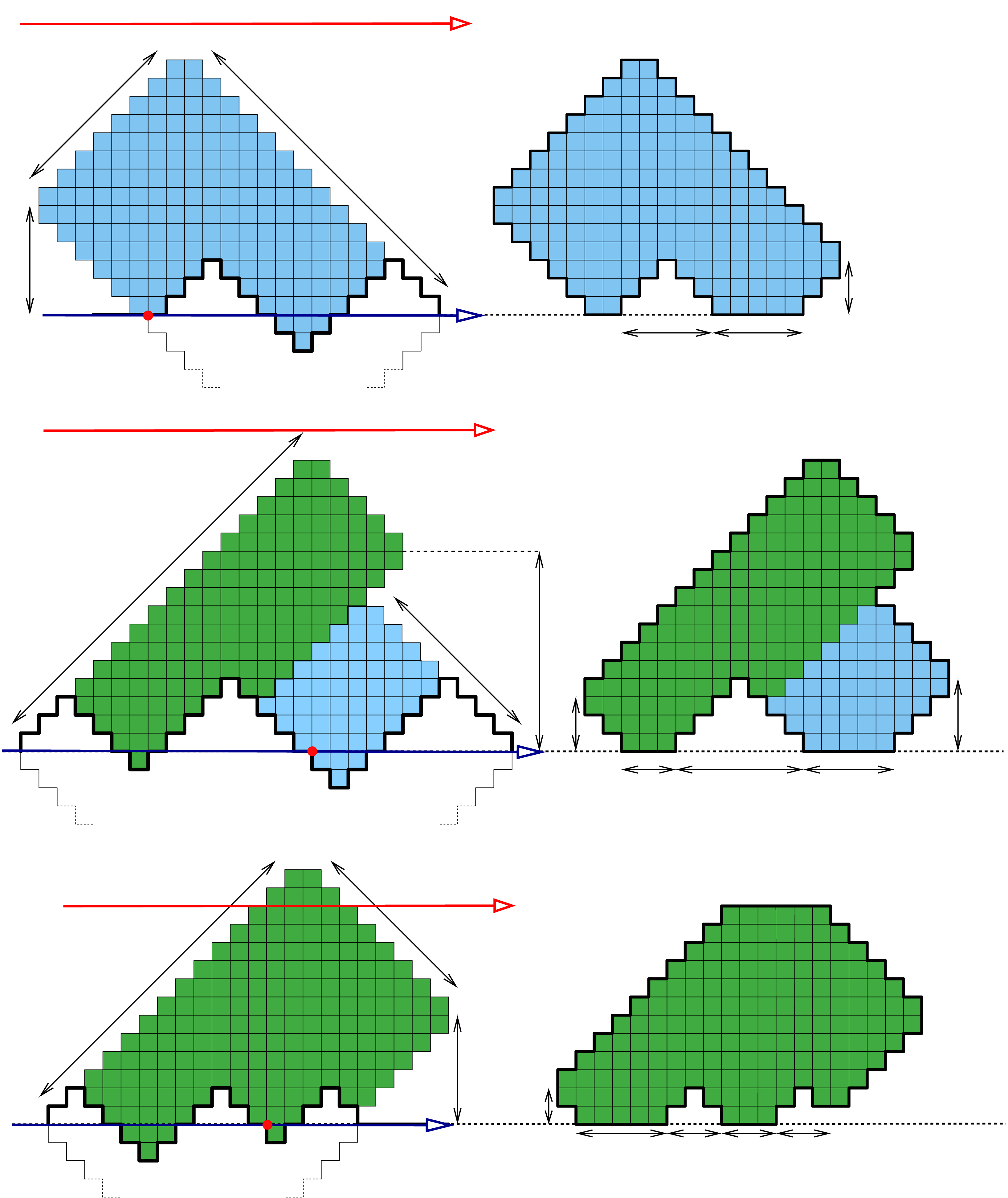}%
\end{picture}%
%
%

\begin{picture}(26153,31089)(407,-30599)
\put(12757,-27271){\makebox(0,0)[lb]{\smash{{\SetFigFont{25}{30.0}{\rmdefault}{\mddefault}{\itdefault}{\color[rgb]{0,0,0}k}%
}}}}
\put(17723,-29632){\makebox(0,0)[lb]{\smash{{\SetFigFont{34}{40.8}{\rmdefault}{\mddefault}{\updefault}{\color[rgb]{0,0,0}$2k_2-1$}%
}}}}
\put(19377,-29632){\makebox(0,0)[lb]{\smash{{\SetFigFont{34}{40.8}{\rmdefault}{\mddefault}{\updefault}{\color[rgb]{0,0,0}$2t_1+1$}%
}}}}
\put(20794,-29632){\makebox(0,0)[lb]{\smash{{\SetFigFont{34}{40.8}{\rmdefault}{\mddefault}{\updefault}{\color[rgb]{0,0,0}$2k_1-1$}%
}}}}
\put(13707,-28215){\makebox(0,0)[lb]{\smash{{\SetFigFont{34}{40.8}{\rmdefault}{\mddefault}{\updefault}{\color[rgb]{0,0,0}$k_3$}%
}}}}
\put(10400,-22546){\rotatebox{315.0}{\makebox(0,0)[lb]{\smash{{\SetFigFont{34}{40.8}{\rmdefault}{\mddefault}{\updefault}{\color[rgb]{0,0,0}$k+t-h-k_1=8$}%
}}}}}
\put(3314,-25144){\rotatebox{45.0}{\makebox(0,0)[lb]{\smash{{\SetFigFont{34}{40.8}{\rmdefault}{\mddefault}{\updefault}{\color[rgb]{0,0,0}$2k+t-h-k_1=14$}%
}}}}}
\put(14647,-18531){\makebox(0,0)[lb]{\smash{{\SetFigFont{34}{40.8}{\rmdefault}{\mddefault}{\updefault}{\color[rgb]{0,0,0}$k_3$}%
}}}}
\put(16773,-20185){\makebox(0,0)[lb]{\smash{{\SetFigFont{34}{40.8}{\rmdefault}{\mddefault}{\updefault}{\color[rgb]{0,0,0}$2t_2+1$}%
}}}}
\put(19135,-20184){\makebox(0,0)[lb]{\smash{{\SetFigFont{34}{40.8}{\rmdefault}{\mddefault}{\updefault}{\color[rgb]{0,0,0}$2k_2-1$}%
}}}}
\put(21734,-20184){\makebox(0,0)[lb]{\smash{{\SetFigFont{34}{40.8}{\rmdefault}{\mddefault}{\updefault}{\color[rgb]{0,0,0}$2t_1+1$}%
}}}}
\put(25749,-18295){\makebox(0,0)[lb]{\smash{{\SetFigFont{34}{40.8}{\rmdefault}{\mddefault}{\updefault}{\color[rgb]{0,0,0}$k_1$}%
}}}}
\put(11104,-15176){\rotatebox{315.0}{\makebox(0,0)[lb]{\smash{{\SetFigFont{34}{40.8}{\rmdefault}{\mddefault}{\updefault}{\color[rgb]{0,0,0}$h+k_1=8$}%
}}}}}
\put(2995,-15224){\rotatebox{45.0}{\makebox(0,0)[lb]{\smash{{\SetFigFont{34}{40.8}{\rmdefault}{\mddefault}{\updefault}{\color[rgb]{0,0,0}$2k+t-h-k_1-1=16$}%
}}}}}
\put(17245,-8846){\makebox(0,0)[lb]{\smash{{\SetFigFont{34}{40.8}{\rmdefault}{\mddefault}{\updefault}{\color[rgb]{0,0,0}$2k_2-1$}%
}}}}
\put(19607,-8846){\makebox(0,0)[lb]{\smash{{\SetFigFont{34}{40.8}{\rmdefault}{\mddefault}{\updefault}{\color[rgb]{0,0,0}$2t_1+1$}%
}}}}
\put(22914,-6956){\makebox(0,0)[lb]{\smash{{\SetFigFont{34}{40.8}{\rmdefault}{\mddefault}{\updefault}{\color[rgb]{0,0,0}$k_1$}%
}}}}
\put(1814,-2940){\rotatebox{45.0}{\makebox(0,0)[lb]{\smash{{\SetFigFont{34}{40.8}{\rmdefault}{\mddefault}{\updefault}{\color[rgb]{0,0,0}$h-k+k_1=8$}%
}}}}}
\put(8505,-2704){\rotatebox{315.0}{\makebox(0,0)[lb]{\smash{{\SetFigFont{34}{40.8}{\rmdefault}{\mddefault}{\updefault}{\color[rgb]{0,0,0}$h+k_1=14$}%
}}}}}
\put(710,-6247){\makebox(0,0)[lb]{\smash{{\SetFigFont{25}{30.0}{\rmdefault}{\mddefault}{\itdefault}{\color[rgb]{0,0,0}k}%
}}}}
\put(14520,-16525){\makebox(0,0)[lb]{\smash{{\SetFigFont{25}{30.0}{\rmdefault}{\mddefault}{\itdefault}{\color[rgb]{0,0,0}k}%
}}}}
\put(15833,-29632){\makebox(0,0)[lb]{\smash{{\SetFigFont{34}{40.8}{\rmdefault}{\mddefault}{\updefault}{\color[rgb]{0,0,0}$2t_2+1$}%
}}}}
\put(10926,-29533){\makebox(0,0)[lb]{\smash{{\SetFigFont{25}{30.0}{\rmdefault}{\mddefault}{\itdefault}{\color[rgb]{0,0,.56}y=0}%
}}}}
\put(11881,-19486){\makebox(0,0)[lb]{\smash{{\SetFigFont{25}{30.0}{\rmdefault}{\mddefault}{\itdefault}{\color[rgb]{0,0,.56}y=0}%
}}}}
\put(12241,-8296){\makebox(0,0)[lb]{\smash{{\SetFigFont{25}{30.0}{\rmdefault}{\mddefault}{\itdefault}{\color[rgb]{0,0,.56}y=0}%
}}}}
\put(12931,-22606){\makebox(0,0)[lb]{\smash{{\SetFigFont{25}{30.0}{\rmdefault}{\mddefault}{\itdefault}{\color[rgb]{1,0,0}y=n}%
}}}}
\put(10216,134){\makebox(0,0)[lb]{\smash{{\SetFigFont{25}{30.0}{\rmdefault}{\mddefault}{\itdefault}{\color[rgb]{1,0,0}y=n}%
}}}}
\put(10839,-10554){\makebox(0,0)[lb]{\smash{{\SetFigFont{25}{30.0}{\rmdefault}{\mddefault}{\itdefault}{\color[rgb]{1,0,0}y=n}%
}}}}
\end{picture}}
\caption{Obtaining the region $\overline{\mathcal{Q}}_{x,h}(k_1,\dots,k_s;t_1,\dots,t_{s-1})$ from the region $\overline{\mathcal{H}}_{x,h}(k_1,\dots,k_s;t_1,\dots,t_{s-1})$: (a) The case when $s=2, k_1=3,k_2=3,t_1=2, x=13, h=11$. (b) The example for $s=3,k_1=4,k_2=4,k_3=3,t_1=2,t_2=1,x=7, h=4$. (c) The example for $s=3, k_1=k_2=k_3=2, t_1=1,t_2=2, x=4, h=-1$.}
\label{Newreverse}
\end{figure}

\medskip

Similar to Theorem \ref{genthm}, we have the following theorem for $\overline{\mathcal{Q}}$-type regions.

\begin{thm}\label{genthm2}
Assume that $s,k_1,\dotsc,k_s,t_1,\dots,t_{s-1}$ are positive integers, and that $M$ is a given $n\times n$ matrix. Assume in addition that $\TAD_{x-h,k_1}^{|h|,n}$ is the truncated Aztec diamond corresponding to the contiguous minor $\det M_{A_1}^{B_1}$ as in Theorem \ref{KWthm}. Then
\begin{equation}\label{geneq2}
\overline{\SM}_{a,b}(k_1,\dotsc,k_s;t_1,\dotsc,t_{s-1})=\Pol(\overline{\mathcal{Q}}_{x,h}(k_1,\dots,k_s;t_1,\dots,t_{s-1})).
\end{equation}
\end{thm}

One readily sees that Theorems \ref{genthm} and \ref{genthm2} verify Conjecture \ref{KWconj}. The proofs of our main theorems will be given in Section 4.

\section{Dodgson Condensation and Kuo Condensation}

This section is devoted to two powerful tools in combinatorics, Dodgson and Kuo condensations.

\medskip

Given a matrix $M$, we denote by $M_{\widehat{a_1,\dots,a_k}}^{\widehat{b_l,\dots,b_1}}$ the matrix obtained from $M$ by removing the rows $a_1,a_2,\dots,a_k$ and the columns $b_1,b_2,\dots,b_l$. We employ the following Dodgson condensation \cite{Dodgson} and its variation in our proof.

\begin{lem}[Dodgson condensation]\label{DJ}
Let $M$ be an $n\times n$ matrix. Then
\begin{equation}
\det M_{\widehat{a}}^{\widehat{c}}\det M_{\widehat{b}}^{\widehat{d}}=\det M \det M_{\widehat{a,b}}^{\widehat{c,d}}+\det M_{\widehat{b}}^{\widehat{c}}\det M_{\widehat{a}}^{\widehat{d}},
\end{equation}
where the indices $a,b,d,c$ appear in counter-clockwise order around the circle. See Figure \ref{diskKuo} for an example.
\end{lem}

The following variation of Dodgson condensation, called  the ``jaw move", is due to Kenyon and Wilson (see the proof of Theorem 7 in  \cite{KW}).
\begin{lem}[Jaw Move]\label{Jaw}
Let $M$ be an $n\times (n+1)$ matrix. Then
\begin{equation}
\det M^{\widehat{e}}\det M^{\widehat{d,f}}_{\widehat{g}}=\det M^{\widehat{d}} \det M^{\widehat{e,f}}_{\widehat{g}}+\det M^{\widehat{f}}\det M^{\widehat{d,e}}_{\widehat{g}},
\end{equation}
where the indices $g,f,e,d$ appear in counter-clockwise order around the circle. The jaw move is illustrated in Figure \ref{disk6a}.
\end{lem}

\begin{figure}\centering
\includegraphics[width=13cm]{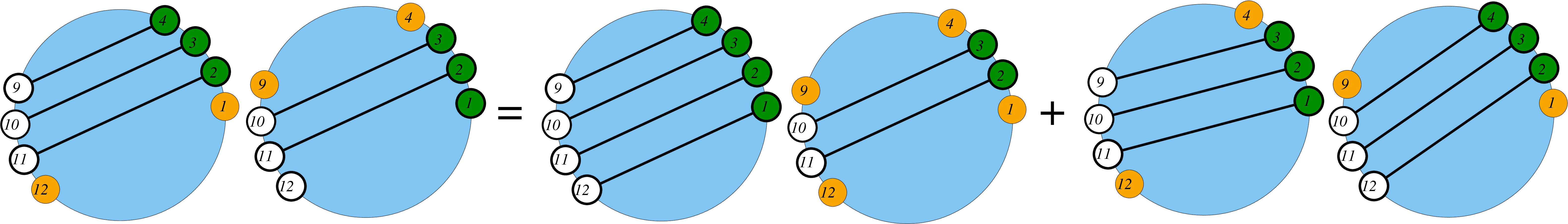}
\caption{Illustration of Dodgson condensation for $M=\Omega_{1,2,3,4}^{12,11,10,9}$, $a=1,$ $b=4,$ $c=12,$ and $d=9$.}
\label{diskKuo}
\end{figure}

\begin{figure}\centering
\includegraphics[width=13cm]{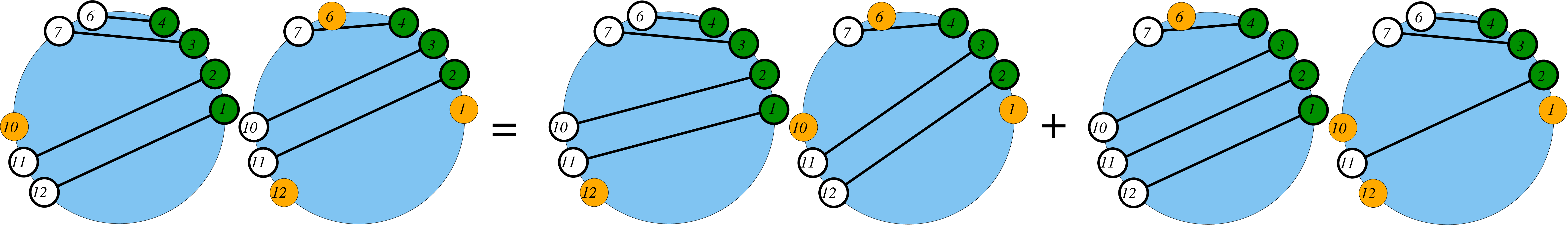}
\caption{Illustration of jaw move for $M=\Omega_{1,2,3,4}^{12,11,10,7,6}$, $d=12,$ $e=10,$ $f=6,$ and  $g=1$.}
\label{disk6a}
\end{figure}

A \emph{perfect matching} of a (simple, finite) graph $G=(V,E)$ is a collection of disjoint edges in $E$ which cover all vertex set $V$. Similar to the case of tilings, we define the weight $\W(G)$ of the graph $G$ to be the sum of the weights of all perfect matchings of $G$, where the \emph{weight} of a perfect matching is the product of weights of its edges.

The \emph{dual graph} of a region $R$ on the square lattice is the graph whose vertices are the unit squares in $R$ and whose edges connect precisely two unit squares sharing and edge. If the dominoes in $R$ are weighted, then each edge in its dual graph $G$ has the same weight as the corresponding domino. This way one can identify the tilings of a region $R$ with the perfect matchings of its dual graph $G$. In particular, we have $\W(R)=\W(G)$.

Kuo \cite{Kuo} proved the following combinatorial interpretations of the Dodgson condensation.

\begin{thm}[Theorem 5.1 in \cite{Kuo}]\label{Kuo1}
Let $G=(V_1,V_2,E)$ be a (weighted) planar bipartite graph with $|V_1|=|V_2|$. Assume that $u,v,w,s$ are four vertices appearing in a cyclic order on a face of $G$. Assume in addition that $u,w\in V_1$ and $v,s\in V_2$. Then
\begin{align}\label{Kuoeq1}
\W(G)W(G-\{u,v,w,s\})=&\W(G-\{u,v\})\W(G-\{w,s\})\notag\\&+\W(G-\{u,s\})\W(G-\{v,w\}).
\end{align}
\end{thm}

\begin{thm}[Theorem 5.2 in \cite{Kuo}]\label{Kuo2}
Let $G=(V_1,V_2,E)$ be a planar bipartite graph with $|V_1|=|V_2|+1$. Assume that $u,v,w,s$ are four vertices appearing in a cyclic order on a face of $G$. Assume in addition that $u,v,w\in V_1$ and $s\in V_2$. Then
\begin{align}\label{Kuoeq2}
\W(G-\{v\})\W(G-\{u,w,s\})=&\W(G-\{u\})\W(G-\{v,w,s\})\notag\\
&+\W(G-\{w\})\W(G-\{u,v,s\}).
\end{align}
\end{thm}

\begin{thm}[Theorem 5.3 in \cite{Kuo}]\label{Kuo3}
Let $G=(V_1,V_2,E)$ be a planar bipartite graph with $|V_1|=|V_2|$. Assume that $u,v,w,s$ are four vertices appearing in a cyclic order on a face of $G$. Assume in addition that $u,v\in V_1$ and $w,s\in V_2$. Then
\begin{align}\label{Kuoeq3}
\W(G-\{u,s\})\W(G-\{v,w\})=&\W(G)\W(G-\{u,v,w,s\})\notag\\
&+\W(G-\{u,w\})\W(G-\{v,s\}).
\end{align}
\end{thm}

We usualy mention the results in Theorems \ref{Kuo1}, \ref{Kuo2} and \ref{Kuo3} as \emph{Kuo condensation}.

\section{Proofs of the main results}

\begin{proof} [Proof of Theorem \ref{genthm}]
We prove the equation (\ref{geneq1}) by induction on $k+s+t$. Recall that $k$ is the cardinality of the index sets $A$ and $B$, $s$ is the number of contiguous components in $B$, and $t$ is the sum of the sizes of the gaps in $B$. The base case is the case when $s=1$, i.e. when $\SM_{a,b}(k_1,\dotsc,k_s;t_1,\dotsc,t_{s-1})$ is a contiguous minor. This case follows directly from Kenyon-Wilson Theorem \ref{KWthm}.

 For the induction step, we assume that $s\geq 2$ and that (\ref{geneq1}) holds for any $\mathcal{Q}$-type regions in which the sum of their $k$-,$s$- and $t$-parameters is strictly less than $k+s+t$.

 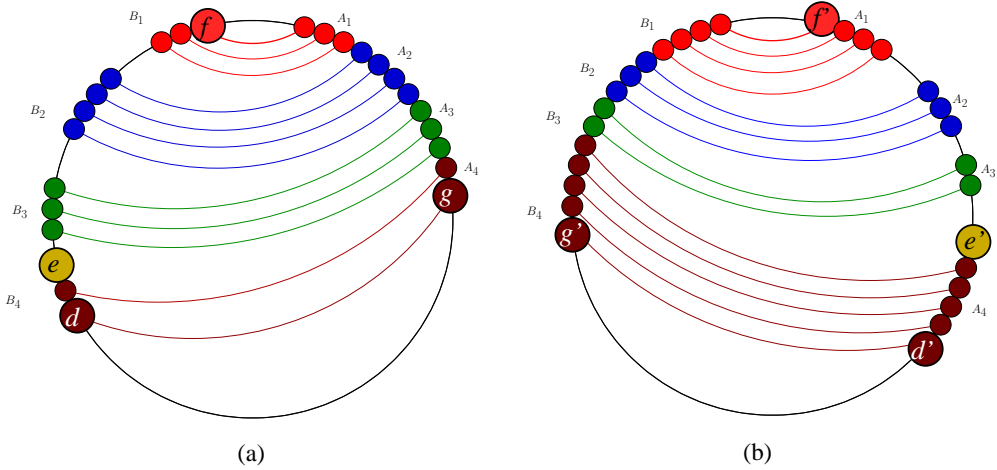
\begin{figure}
 \setlength{\unitlength}{3947sp}%
\begingroup\makeatletter\ifx\SetFigFont\undefined%
\gdef\SetFigFont#1#2#3#4#5{%
  \reset@font\fontsize{#1}{#2pt}%
  \fontfamily{#3}\fontseries{#4}\fontshape{#5}%
  \selectfont}%
\fi\endgroup%

\resizebox{13cm}{!}{
 \begin{picture}(0,0)%
\includegraphics{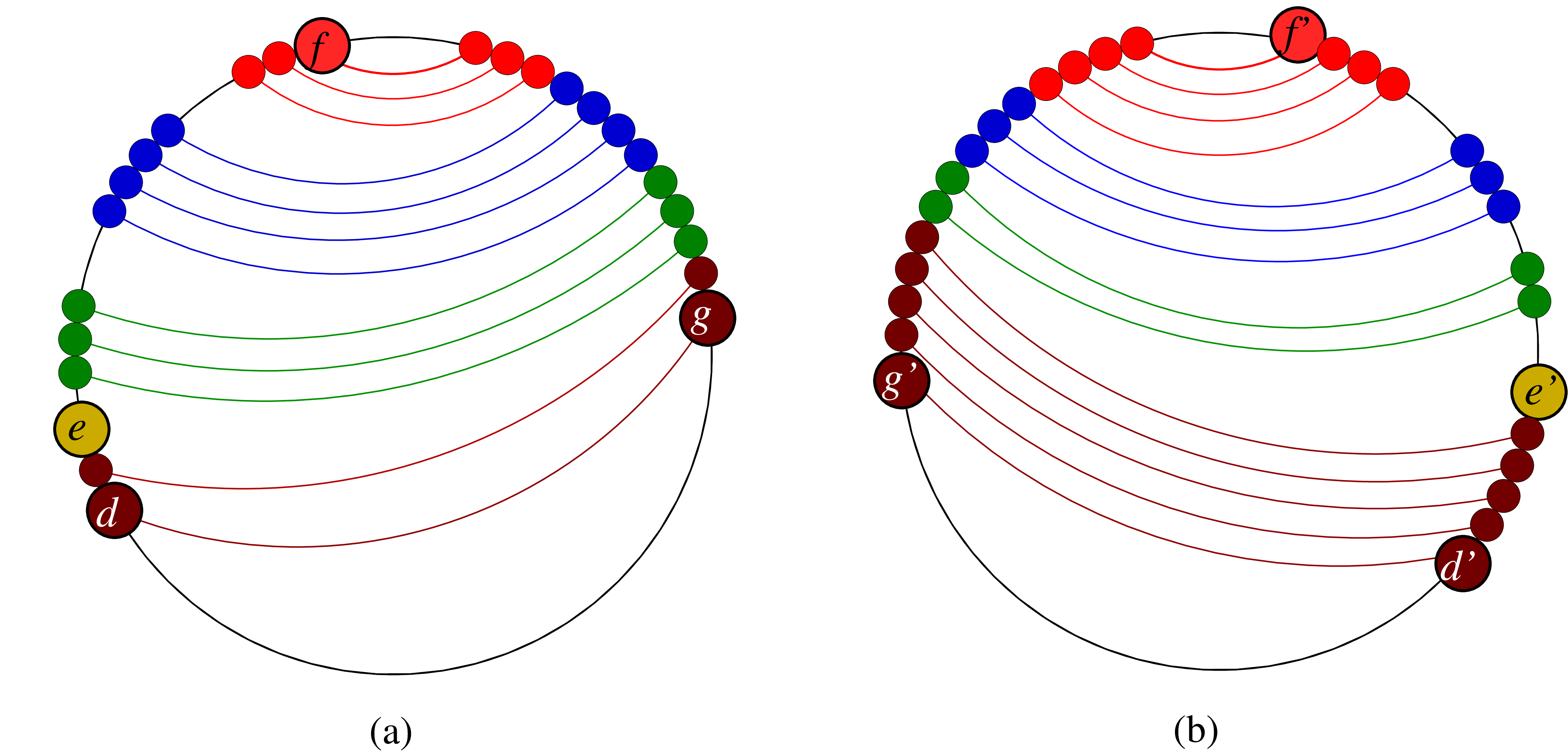}%
\end{picture}%
%
%

\begin{picture}(33063,15933)(846,-17573)
\put(11976,-2479){\makebox(0,0)[lb]{\smash{{\SetFigFont{41}{49.2}{\rmdefault}{\mddefault}{\itdefault}{\color[rgb]{0,0,0}$A_1$}%
}}}}
\put(13986,-3589){\makebox(0,0)[lb]{\smash{{\SetFigFont{41}{49.2}{\rmdefault}{\mddefault}{\itdefault}{\color[rgb]{0,0,0}$A_2$}%
}}}}
\put(15396,-5509){\makebox(0,0)[lb]{\smash{{\SetFigFont{41}{49.2}{\rmdefault}{\mddefault}{\itdefault}{\color[rgb]{0,0,0}$A_3$}%
}}}}
\put(16221,-7519){\makebox(0,0)[lb]{\smash{{\SetFigFont{41}{49.2}{\rmdefault}{\mddefault}{\itdefault}{\color[rgb]{0,0,0}$A_4$}%
}}}}
\put(4986,-2404){\makebox(0,0)[lb]{\smash{{\SetFigFont{41}{49.2}{\rmdefault}{\mddefault}{\itdefault}{\color[rgb]{0,0,0}$B_1$}%
}}}}
\put(1086,-8974){\makebox(0,0)[lb]{\smash{{\SetFigFont{41}{49.2}{\rmdefault}{\mddefault}{\itdefault}{\color[rgb]{0,0,0}$B_3$}%
}}}}
\put(1686,-5584){\makebox(0,0)[lb]{\smash{{\SetFigFont{41}{49.2}{\rmdefault}{\mddefault}{\itdefault}{\color[rgb]{0,0,0}$B_2$}%
}}}}
\put(861,-12004){\makebox(0,0)[lb]{\smash{{\SetFigFont{41}{49.2}{\rmdefault}{\mddefault}{\itdefault}{\color[rgb]{0,0,0}$B_4$}%
}}}}
\put(22116,-2629){\makebox(0,0)[lb]{\smash{{\SetFigFont{41}{49.2}{\rmdefault}{\mddefault}{\itdefault}{\color[rgb]{0,0,0}$B_1$}%
}}}}
\put(20106,-4159){\makebox(0,0)[lb]{\smash{{\SetFigFont{41}{49.2}{\rmdefault}{\mddefault}{\itdefault}{\color[rgb]{0,0,0}$B_2$}%
}}}}
\put(18336,-8989){\makebox(0,0)[lb]{\smash{{\SetFigFont{41}{49.2}{\rmdefault}{\mddefault}{\itdefault}{\color[rgb]{0,0,0}$B_4$}%
}}}}
\put(29297,-2374){\makebox(0,0)[lb]{\smash{{\SetFigFont{41}{49.2}{\rmdefault}{\mddefault}{\itdefault}{\color[rgb]{0,0,0}$A_1$}%
}}}}
\put(32604,-5209){\makebox(0,0)[lb]{\smash{{\SetFigFont{41}{49.2}{\rmdefault}{\mddefault}{\itdefault}{\color[rgb]{0,0,0}$A_2$}%
}}}}
\put(33549,-7571){\makebox(0,0)[lb]{\smash{{\SetFigFont{41}{49.2}{\rmdefault}{\mddefault}{\itdefault}{\color[rgb]{0,0,0}$A_3$}%
}}}}
\put(33246,-12289){\makebox(0,0)[lb]{\smash{{\SetFigFont{41}{49.2}{\rmdefault}{\mddefault}{\itdefault}{\color[rgb]{0,0,0}$A_4$}%
}}}}
\put(18966,-5779){\makebox(0,0)[lb]{\smash{{\SetFigFont{41}{49.2}{\rmdefault}{\mddefault}{\itdefault}{\color[rgb]{0,0,0}$B_3$}%
}}}}
\end{picture}}

 \caption{(a) How to apply the Jaw Move to a $\SM$-minor; the picture corresponds to the minor in Figure \ref{multigap}(a). (b) How to apply the Jaw Move to a $\overline{\SM}$-minor; the picture corresponds to the minor in Figure \ref{multigap}(b).}\label{SMKuo}
 \end{figure}

We apply the Jaw Move in Lemma \ref{Jaw} to the $k\times(k+1)$ matrix $M_{A}^{B\cup\{b+k-k_s+t-1\}}$ with $d=b+k+t-1,e=b+k-k_s+t-1,f=b,g=a$ (see Figure \ref{SMKuo}(a)), and obtain
\begin{align}\label{genminoreq1}
\SM_{a,b}&(k_1,\dotsc,k_s;t_1,\dotsc,t_{s-1}) \SM_{a+1,b+1}(k_1-1,\dotsc,k_s;t_1,\dotsc,t_{s-1}-1)
=\notag\\&\SM_{a,b}(k_1,\dotsc,k_s;t_1,\dotsc,t_{s-1}-1)\SM_{a+1,b+1}(k_1-1,\dotsc,k_s;t_1,\dotsc,t_{s-1})\notag\\
&+\SM_{a,b+1}(k_1-1,\dotsc,k_s+1;t_1,\dotsc,t_{s-1}-1)\SM_{a+1,b}(k_1,\dotsc,k_s-1;t_1,\dotsc,t_{s-1}).
\end{align}
Here we understand that
\begin{equation}
\SM_{a,b}(k_1,\dots,k_{s-1},0;t_1,\dots,t_{s-1})\equiv \SM_{a,b}(k_1,\dots,k_{s-1};t_1,\dots,t_{s-2}),
\end{equation}
\begin{equation}
\SM_{a,b}(0,k_2\dots,k_{s};t_1,\dots,t_{s-1})\equiv \SM_{a,b+t_1} (k_2,\dots,k_{s};t_2,\dots,t_{s-1}),
\end{equation}
and
\begin{equation}
\SM_{a,b} (k_1,\dots,k_s;t_1,\dots,t_{s-2},0)\equiv \SM_{a,b}(k_1,\dots,k_{s-2},k_{s-1}+k_{s};t_1,\dots,t_{s-2}).
\end{equation}
We note that the sum of the $k$-,$s$- and $t$-parameters in any minors in recurrence (\ref{genminoreq1}), except for the first one, are strictly less than $k+s+t$. Indeed, in each of these 5 minors, the number of components in $B$ is not increasing, i.e. the $s$-parameter is at most $s$, and at least one of the $k$- and $t$-parameters is reduced by 1 unit.

Moreover, the $s$-parameter only decreases when either the $k_1$-parameter or $t_{s-1}$-parameter becomes $0$, i.e. either the first contiguous component $B_1$ of $B$ disappears or the last two components $B_{s-1}$ and $B_s$ merge into a single component.

\medskip

To prove (\ref{geneq1}), we will use Kuo condensation to show that the tiling polynomial\\ $\Pol(\mathcal{Q}_{x,h}(k_1,\dots,k_s;t_1,\dots, $ $t_{s-1}))$ satisfies the same recurrence.  There are three cases to distinguish, based on the type of the region $\mathcal{Q}:=\mathcal{Q}_{x,h}(k_1,\dots,k_s;t_1,\dots,t_{s-1})$.

To make sure our process runs smoothly, we assume by convention in the rest of this proof  that
\begin{equation}
\mathcal{Q}_{x,h}(k_1,\dots,k_{s-1},0;t_1,\dots,t_{s-1})\equiv\mathcal{Q}_{x,h}(k_1,\dots,k_{s-1};t_1,\dots,t_{s-2}),
\end{equation}
\begin{equation}
\mathcal{Q}_{x,h}(0,k_2\dots,k_{s};t_1,\dots,t_{s-1})\equiv\mathcal{Q}_{x-k_2,h-k_2-t_1}(k_2,\dots,k_{s};t_2,\dots,t_{s-1}),
\end{equation}
\begin{equation}
\mathcal{Q}_{x,h}(k_1,\dots,k_s;t_1,\dots, t_{s-2},0)\equiv\mathcal{Q}_{x,h}(k_1,\dots,k_{s-2},k_{s-1}+k_{s};t_1,\dots,t_{s-2}).
\end{equation}

\begin{figure}\centering
\includegraphics[width=13cm]{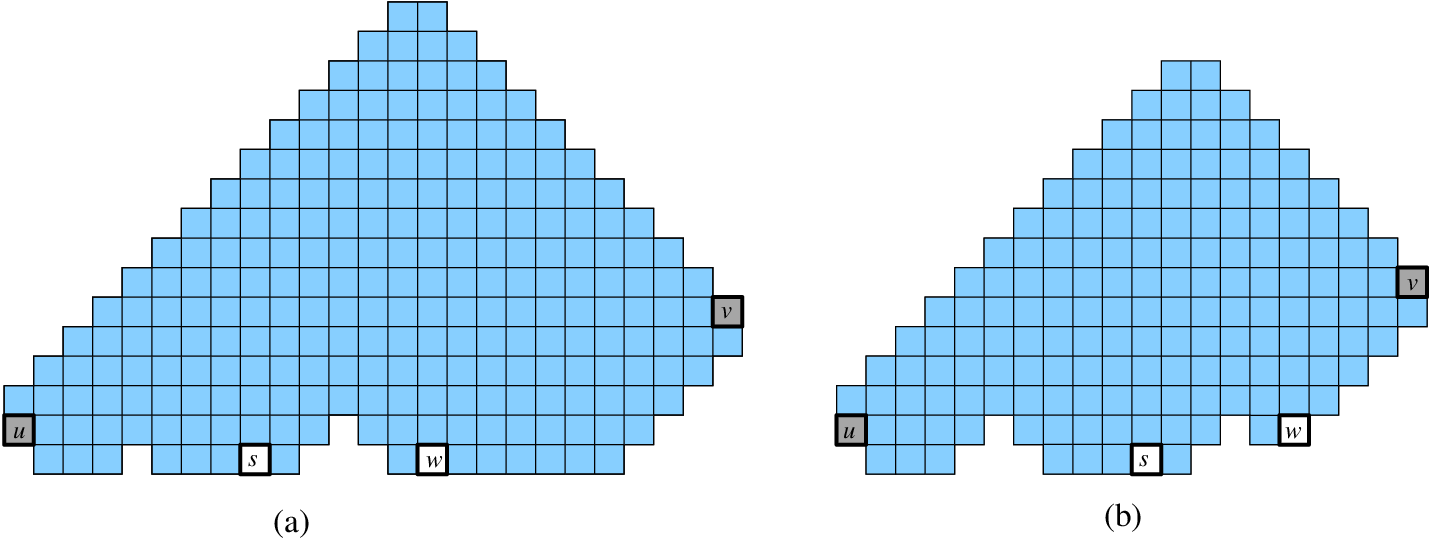}
\caption{How we apply Kuo condensation in the case when $1\leq t \leq h-k$.}
\label{Kuogeneral2}
\end{figure}

\begin{figure}\centering
\includegraphics[width=13cm]{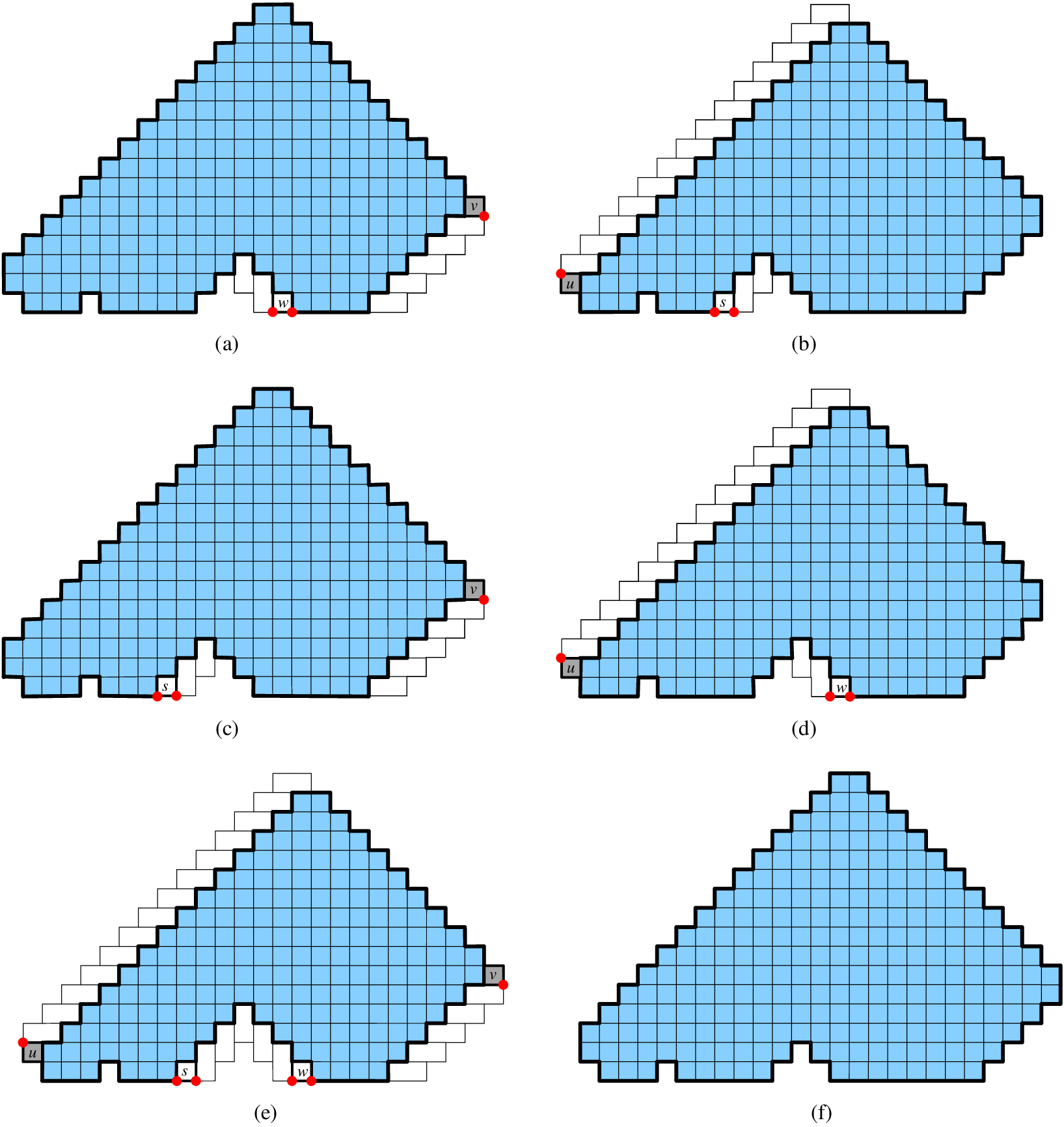}
\caption{Obtaining the recurrence for $\Pol(\mathcal{Q}_{x,h}(k_1,\dots,k_s;t_1,\dots, $ $t_{s-1}))$ in the case when $t\leq h-k$.}
\label{Kuogeneral2b}
\end{figure}

\medskip

\textit{Case 1. $t \leq h-k$.}\\

We color the unit squares on the square lattice black and white so that two adjacent unit squares have different colors. Without loss of generality, we assume that the Aztec rectangle $\AR_{x-h,0}^{h+k_1,h-k+k_1}$ (in the definition of Type 1 $\mathcal{Q}$-regions) has white unit squares along its northwest boundary.

By Remark \ref{heightrmk}, the top of the region $\mathcal{H}=\mathcal{H}_{x,h}(k_1,\dots,k_s-1;t_1,\dots,t_{s-1})$ is always below or on the line $y=n$.

We apply Kuo's Theorem \ref{Kuo3} to the dual graph of $G$ the region  $\mathcal{Q}_{x,h}(k_1,\dots,k_s-1;t_1,\dots,t_{s-1})$.  We pick the four vertices $u,v,w,s$ as in  Figure \ref{Kuogeneral2}(a) (for\footnote{Our arguments in this proof work regardless the value of $x$. Thus, to make our illustrating figures simple, we do not give any particular value for $x$ the figures.} $s=3, k_1=2, k_2=1, k_3=3, t_1=1, t_2=2, h=14$) and Figure \ref{Kuogeneral2}(b) (for $s=3,k_1=2,k_2=2, k_3=3, t_1=1, t_2=2, h=12$). The vertices $u,v$ correspond to the shaded unit squares, the vertices $w,s$ correspond to the white unit squares with bold boundary in those figures. In particular, the vertices $u$ and $v$ correspond to the leftmost and the rightmost black unit square in the region; $s$ corresponds to the  white unit square that is $2(k_2+\dotsc+k_{s-1}+t_1+\dotsc+t_{s-1})-1$ units to the right of the leftmost unit square on the base, and $w$ corresponds to the white unit square that is $2(k_2+\dotsc+k_{s}+t_1+\dotsc+t_{s-1})-1$ units to the right of the leftmost unit square on the base if such vertex exists (as in Figure \ref{Kuogeneral2}(a)), otherwise we pick the $w$-square as the lowest white unit square on the stair going southwest from the unit square corresponding to $v$ (as in Figure \ref{Kuogeneral2}(b)).

Consider the region corresponding to the graph $G-\{v,w\}$. It has several dominoes, which are forced to be in any tiling of the region. By removing these dominoes, we get back the region $\mathcal{Q}_{x,h}(k_1,\dots,k_s;t_1,\dots,$ $t_{s-1})$ (see the shaded region restricted by the bold contour in Figure \ref{Kuogeneral2b}(a) for $s=3,k_1=2, k_2=1, k_3=3, t_1=1, t_2=2, h=14$).

Similarly, by removing forced dominoes from the regions corresponding to $G-\{u,s\}$, $G-\{v,s\}$, $G-\{u,w\},$ and $G-\{u,v,w,s\}$, we get the regions $\mathcal{Q}_{x+1,h}(k_1-1,\dots,k_s;t_1,\dots,t_{s-1}-1)$, $\mathcal{Q}_{x,h}(k_1,\dots,k_s;t_1,\dots,t_{s-1}-1)$, \\ $\mathcal{Q}_{x+1,h}(k_1-1,\dots,k_s;t_1,\dots,t_{s-1})$ and $\mathcal{Q}_{x+1,h}(k_1-1,\dots,k_s+1;t_1,\dots,t_{s-1}-1)$, respectively (see Figures \ref{Kuogeneral2b}(b)--(e)).  Theorem \ref{Kuo3} and Figure \ref{Kuogeneral2b} tell us that the product of the weights of the two regions on the top is equal to the product of the weights of the two regions in the middle, plus the product of the weights of the two regions on the bottom. Equivalently, we have
\begin{align}\label{gentilingeq0}
C_1&C_2\W(\mathcal{Q}_{x,h}(k_1,\dots,k_s;t_1,\dots,t_{s-1}))\W(\mathcal{Q}_{x+1,h}(k_1-1,\dots,k_s;t_1,\dots,t_{s-1}-1))=\notag\\
&C_3C_4\W(\mathcal{Q}_{x,h}(k_1,\dots,k_s;t_1,\dots,t_{s-1}-1)) \W(\mathcal{Q}_{x+1,h}(k_1-1,\dots,k_s;t_1,\dots,t_{s-1}))\notag\\
&+ C_5C_6\W(\mathcal{Q}_{x+1,h}(k_1-1,\dots,k_s+1;t_1,\dots,t_{s-1}-1))\W(\mathcal{Q}_{x,h}(k_1,\dots,k_s-1;t_1,\dots,t_{s-1})),
\end{align}
where $C_i$ is the product of weights of forced dominoes in the region corresponding to the $i$-th graph (from left to right) in the equation (\ref{Kuoeq3}) of  Theorem \ref{Kuo3}. (Of course in this case, we have $C_6=1$.)

Comparing the covering monomial of the region corresponding to $G$ to the covering monomials of the other regions in the equation (\ref{gentilingeq0}), we obtain
\begin{equation}\label{gentilingeq1a}
\frac{\F(\mathcal{Q}_{x,h}(k_1,\dots,k_s;t_1,\dots,t_{s-1}))}{\F(\mathcal{Q}_{x,h}(k_1,\dots,k_s-1;t_1,\dots,t_{s-1}))}=\frac{C_1}{D_vD_w},
\end{equation}
\begin{equation}\label{gentilingeq1b}
\frac{\F(\mathcal{Q}_{x+1,h}(k_1-1,\dots,k_s;t_1,\dots,t_{s-1}-1))}{\F(\mathcal{Q}_{x,h}(k_1,\dots,k_s-1;t_1,\dots,t_{s-1}))}=\frac{C_2}{D_uD_s},
\end{equation}
\begin{equation}\label{gentilingeq1c}
\frac{\F(\mathcal{Q}_{x,h}(k_1,\dots,k_s;t_1,\dots,t_{s-1}-1))}{\F(\mathcal{Q}_{x,h}(k_1,\dots,k_s-1;t_1,\dots,t_{s-1}))}=\frac{C_3}{D_vD_s},
\end{equation}
\begin{equation}\label{gentilingeq1d}
\frac{\F(\mathcal{Q}_{x+1,h}(k_1-1,\dots,k_s;t_1,\dots,t_{s-1}))}{\F(\mathcal{Q}_{x,h}(k_1,\dots,k_s-1;t_1,\dots,t_{s-1}))}=\frac{C_4}{D_uD_w},
\end{equation}
\begin{equation}\label{gentilingeq1e}
\frac{\F(\mathcal{Q}_{x+1,h}(k_1-1,\dots,k_s+1;t_1,\dots,t_{s-1}-1))}{\F(\mathcal{Q}_{x,h}(k_1,\dots,k_s-1;t_1,\dots,t_{s-1}))}=\frac{C_5}{D_uD_vD_wD_s},
\end{equation}
where $D_u$ (resp., $D_v,D_w,D_s$) is the product of  those  terms $v_{x,y}$,  which  correspond to the lattice points $(x,y)$ adjacent the  $u$-square (resp., $v$-square, $w$-square, $s$-square), and are not the central of the long side of any forced domino (illustrated by the red dots in Figures \ref{Kuogeneral2b}(a)--(e)).

By (\ref{gentilingeq0})--(\ref{gentilingeq1e}), we get
\begin{align}\label{gentilingeq1}
\Pol(\mathcal{Q}_{x,h}&(k_1,\dots,k_s;t_1,\dots,t_{s-1}))\Pol(\mathcal{Q}_{x+1,h}(k_1-1,\dots,k_s;t_1,\dots,t_{s-1}-1))=\notag\\
&\Pol(\mathcal{Q}_{x,h}(k_1,\dots,k_s;t_1,\dots,t_{s-1}-1)) \Pol(\mathcal{Q}_{x+1,h}(k_1-1,\dots,k_s;t_1,\dots,t_{s-1}))\notag\\
&+ \Pol(\mathcal{Q}_{x+1,h}(k_1-1,\dots,k_s+1;t_1,\dots,t_{s-1}-1))\Pol(\mathcal{Q}_{x,h}(k_1,\dots,k_s-1;t_1,\dots,t_{s-1})).
\end{align}
 This means that $\SM_{a,b}(k_1,\dotsc,k_s;t_1,\dotsc,t_{s-1})$ and $\Pol(\mathcal{Q}_{x,h}(k_1,\dots,k_s;t_1,\dots,t_{s-1}))$ satisfy the same recurrence in this case. Finally, we note that the latter regions in recurrence (\ref{gentilingeq1}) may \emph{not} be in the same type as the first one, however, our induction procedure still works, as long as the sum of $k$-,$s$- and $t$-parameters in these regions is strictly less than $k+s+t$. This fact can be verified in the same way  as we did for the $\SM$-minors in recurrence (\ref{genminoreq1}) above.

\begin{figure}\centering
\includegraphics[width=14cm]{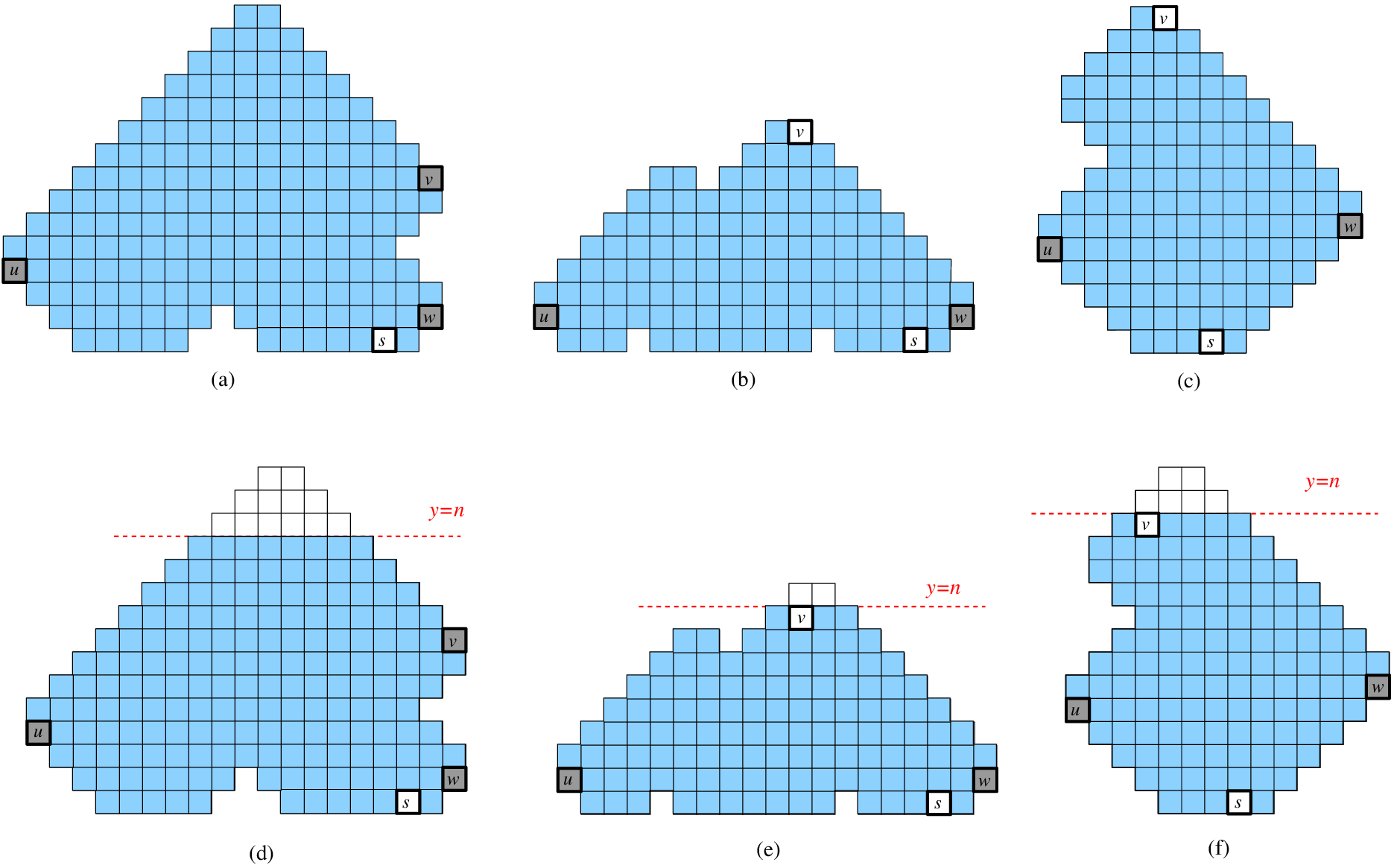}
\caption{How we apply Kuo condensation when $h\geq 0^+$ and $ t > h-k$.}
\label{Kuogeneral2n}
\end{figure}

\begin{figure}\centering
\includegraphics[width=13cm]{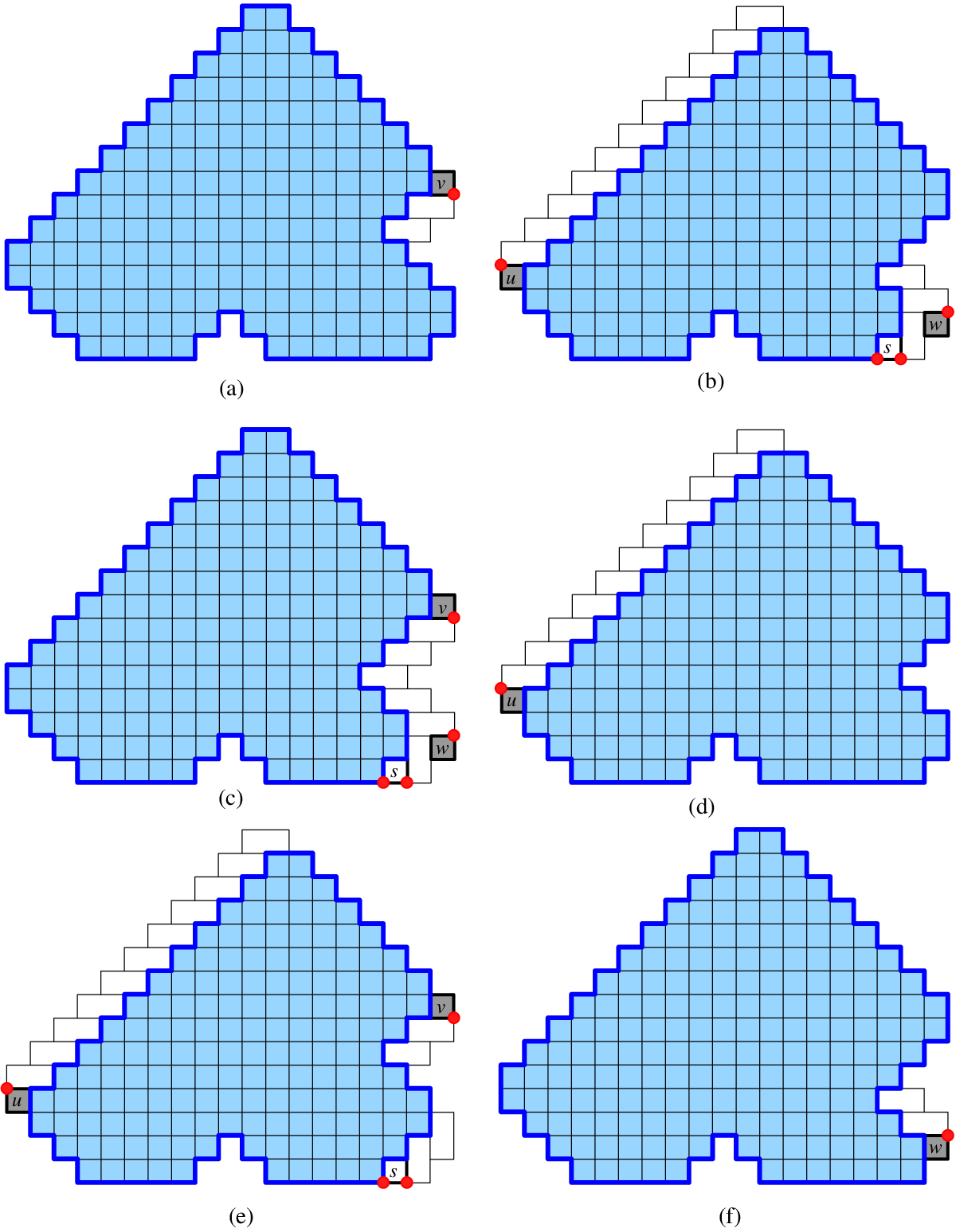}
\caption{Obtaining the recurrence for $\Pol(\mathcal{Q}_{x,h}(k_1,\dots,k_s;t_1,\dots, $ $t_{s-1}))$ in case when $1\leq t+k-h\leq k_1$.}
\label{Kuogeneral2m}
\end{figure}


\medskip

\textit{Case 2. $h\geq 0^+$ and $ t > h-k$.}\\

 We work first to the case the top of $\mathcal{H}$ is below or on the line $y=n$.

We consider first the subcase when $1\leq t+k-h\leq k_1$. This corresponds to the case  $\AD_2$ stays inside $\AD_1$. We apply Kuo's Theorem \ref{Kuo2} to the graph $G$ of the region $R$ that is the union of the two regions $\mathcal{Q}_{x,h}(k_1,\dots,k_s;t_1,\dots,t_{s-1})$ and $\mathcal{Q}_{x+1,h}(k_1-1,\dots,k_s;t_1,\dots,t_{s-1})$ as  in Figure \ref{Kuogeneral2n}(a) for $s=3,k_1=4,k_2=k_3=2, t_1=2, t_2=3, h=11$ (these two regions $\mathcal{Q}_{x,h}(k_1,\dots,k_s;t_1,\dots,t_{s-1})$ and $\mathcal{Q}_{x+1,h}(k_1-1,\dots,k_s;t_1,\dots,t_{s-1})$ are shown later in Figures \ref{Kuogeneral2m}(a) and (d), respectively).  We color the unit squares of $R$ like a chessboard, so that the unit squares on the northwest side of $\AD_1$ are white. The vertex $u$ corresponds to  the leftmost black unit square in $R$, $v$ corresponds to the last black unit square on the stair going southeast from the top, $w$ corresponds to the leftmost black unit square on the stair going northeast from the bottom, and the unit square corresponding to $s$ is the rightmost white unit square on the base.  We obtain the regions $\mathcal{Q}_{x,h}(k_1,\dots,k_s;t_1,\dots,t_{s-1})$, $\mathcal{Q}_{x+1,h}(k_1-1,\dots,k_s;t_1,\dots,t_{s-1}-1)$,
$\mathcal{Q}_{x,h}(k_1,\dots,k_s;t_1,\dots,t_{s-1}-1)$, $\mathcal{Q}_{x+1,h}(k_1-1,\dots,k_s;t_1,\dots,t_{s-1})$,
$\mathcal{Q}_{x+1,h}(k_1-1,\dots,k_s+1;t_1,\dots,t_{s-1}-1)$,  and $\mathcal{Q}_{x,h}(k_1,\dots,k_s-1;t_1,\dots,t_{s-1})$ by removing forced dominoes from the regions corresponding to the graphs $G-\{v\}$, $G-\{u,w,s\}$, $G-\{v,w,s\}$, $G-\{u\}$, $G-\{u,v,s\}$ and $G-\{w\}$, respectively (see Figures \ref{Kuogeneral2m}(a)--(f) respectively). We get again the equation  (\ref{gentilingeq0}) as in Case 1, where $C_i$ is now the product of weights of forced dominoes in the region corresponding to the $i$-th graph in the equation (\ref{Kuoeq2}) of Theorem \ref{Kuo2}. By comparing the covering the monomials of the regions in  (\ref{gentilingeq0}) to the covering monomial of $R$, we get also  (\ref{gentilingeq1}).

\begin{figure}\centering
\includegraphics[width=14cm]{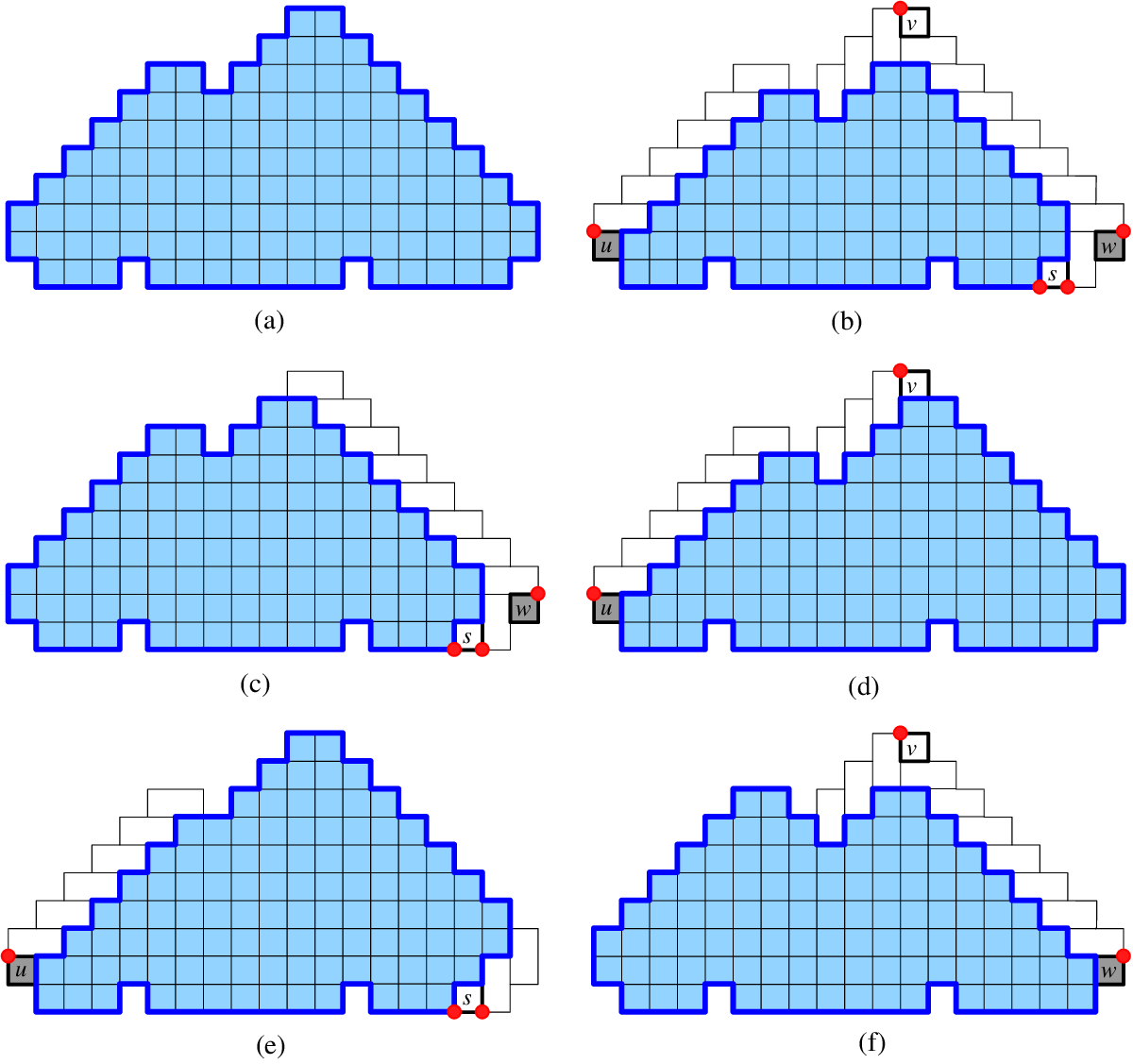}
\caption{Obtaining the recurrence for $\Pol(\mathcal{Q}_{x,h}(k_1,\dots,k_s;t_1,\dots, $ $t_{s-1}))$ in case when $t+k-h> k_1$.}
\label{Kuogeneral2d}
\end{figure}

Next, we investigate the subcase when $t+k-h>k_1$, then $\AD_2$ does not stay inside $\AD_1$ any more. In this case, Theorem \ref{Kuo1} has been used for the dual graph of $G$ of the region $\mathcal{Q}:=\mathcal{Q}_{x,h}(k_1,\dots,k_s;t_1,\dots,t_{s-1})$ with the four vertices $u,v,w,s$ corresponding to the unit squares of the same label as in Figure \ref{Kuogeneral2n}(b) for $s=4,k_1=2,k_2=k_3=1, k_4=2, t_1=1, t_2=3,t_3=2, h=6$, and Figure \ref{Kuogeneral2n}(c) for $s=2, k_1=5,k_2=6,t_1=2,h=3$. In particular, the $u$-square and $w$-square are the leftmost  and the rightmost black unit squares of $\mathcal{Q}$, the $w$-square is the white unit square on the top of $\AD_2$, and the $s$-square is still the rightmost white unit square on the base. By removing the forced dominoes, we respectively transform  the regions corresponding to the graphs $G-\{u,v,w,s\}$, $G-\{w,s\}$, $G-\{u,v\}$, $G-\{u,s\}$ and $G-\{v,w\}$ into the regions in the equation (\ref{gentilingeq0})
  (see Figures \ref{Kuogeneral2d} (b)--(f), respectively). By the same process as in the previous cases, we also obtain (\ref{gentilingeq1}).

 \medskip

 We now consider the situation when the top of $\mathcal{H}$ is above the line $y=n$. It is easy to see that when $1\leq t+k-h\leq k_1-1$, our region looks like Figure \ref{Kuogeneral2n}(d). Then the Kuo's condensation still works here the same as in Figures \ref{Kuogeneral2n}(a) and \ref{Kuogeneral2m}.  In the case when $k-h\geq k_1$ (i.e. $AD_1$ stays  inside $AD_2$), our region looks like \ref{Kuogeneral2n}(f); if $t+k-h\geq k_1>k-h$ (i.e. $AD_1$ and $AD_2$ are not inside each other), our region has the structure as in Figure \ref{Kuogeneral2n}(e). In this case, we pick  $u,w,s$ similarly as in Figures \ref{Kuogeneral2n}(b) and (c), however, the $v$-square is now the second square of the truncated top of $AD_2$ if it is cut off by the line $y=n$. Then Kuo condensation still works the same as in Figure \ref{Kuogeneral2d}.

 \medskip

\textit{Case 3. $h \leq 0^-$.}\\

\begin{figure}\centering
\includegraphics[width=13cm]{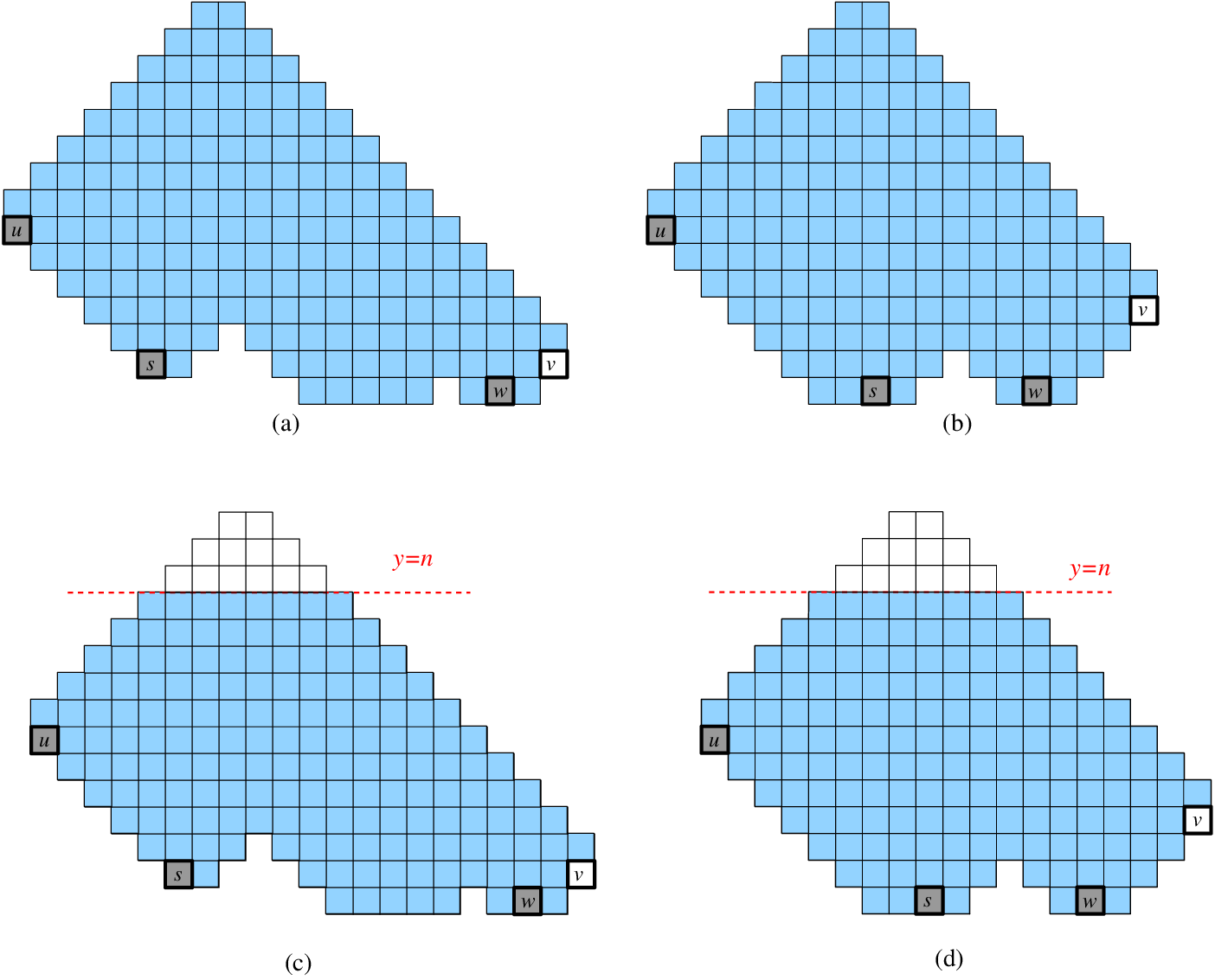}
\caption{How we apply Kuo condensation when $h\leq 0^-$.}
\label{Newnegative}
\end{figure}

\begin{figure}\centering
\includegraphics[width=13cm]{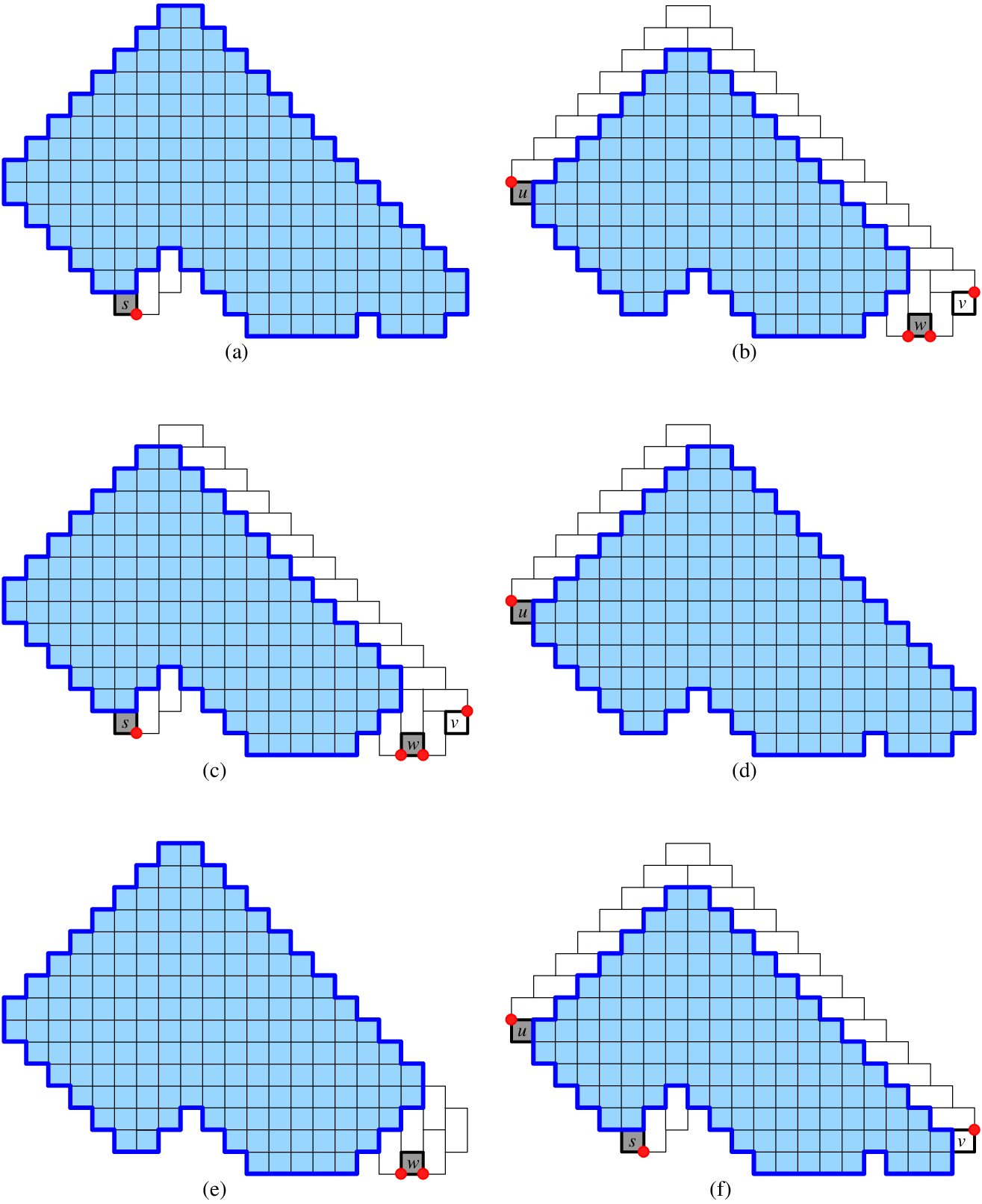}
\caption{Obtaining the recurrence when $h\leq 0^-$.}
\label{Newnegative2}
\end{figure}

Similar to the above cases, we assume that the square lattice has a chessboard coloring, so that the Aztec rectangle $\AR_{x-h+t,k}^{k+t-h-k_1,2k+t-h-k_1}$ (in the definition of the type-3 $\mathcal{Q}$ regions) has white unit squares on the northwest boundary.

We assume first that the height of $\mathcal{H}$ not greater than $n$. We apply  Kuo's Theorem \ref{Kuo2} to the dual graph $G$ of the union $R$ of the two regions $\mathcal{Q}_{x,h}(k_1,\dots,k_s; $ $t_1,\dots,t_{s-1})$ and $\mathcal{Q}_{x+1,h}(k_1-1,\dots,k_s;t_1,\dots,t_{s-1})$ as in Figure \ref{Newnegative}(a) for $s=3,k_1=4,k_2=1,k_3=2,t_1=2,t_2=1,h=-2$ (these two regions $\mathcal{Q}_{x,h}(k_1,\dots,k_s; $ $t_1,\dots,t_{s-1})$ and $\mathcal{Q}_{x+1,h}(k_1-1,\dots,k_s;t_1,\dots,t_{s-1})$ themselves are shown later in Figures \ref{Newnegative2}(a) and (d)),  and Figure \ref{Newnegative}(b) for $s=2,k_1=3,k_2=4,t_1=1,h=-3$. In this case, the  $u$-square is the leftmost black unit square in $R$, the $v$-square is the rightmost white unit square, the $w$-square is  the rightmost black unit square on the base, and the $s$-square is the black unit square that is $2(k_1+\dots+k_{s-1}+t_1+\dotsc+t_{s-1})-1 $ units to the right of the leftmost unit square on the base if such vertex exists (as in Figure \ref{Newnegative}(b)), otherwise we pick it as the last black unit square on the stair going southeast from the $u$-square (as in Figure \ref{Newnegative}(a)). By Figures \ref{Newnegative2}(a)--(f), we get the regions $\mathcal{Q}_{x,h}(k_1,\dots,k_s;t_1,\dots,t_{s-1})$, $\mathcal{Q}_{x+1,h}(k_1-1,\dots,k_s;t_1,\dots,t_{s-1}-1)$, $\mathcal{Q}_{x,h}(k_1,\dots,k_s;t_1,\dots,t_{s-1}-1)$, $\mathcal{Q}_{x+1,h}(k_1-1,\dots,k_s;t_1,\dots,t_{s-1})$, $\mathcal{Q}_{x+1,h}(k_1-1,\dots,k_s+1;t_1,\dots,t_{s-1}-1)$, and $\mathcal{Q}_{x,h}(k_1,\dots,k_s-1;t_1,\dots,t_{s-1})$ by removing forced dominoes from  the regions corresponding to the graphs $G-\{s\}$, $G-\{u,v,w\}$, $G-\{v,w,s\}$, $G-\{u\}$, $G-\{w\}$, and $G-\{u,v,s\}$, respectively. Then (\ref{gentilingeq1}) follows from  Theorem \ref{Kuo2} in the same way as in the previous cases.

Finally, if the top of $\mathcal{H}$ is above the line $y=n$, i.e. the top of $\AR_{x-h+t,k}^{k+t-h-k_1,2k+t-h-k_1}$ is above the line $y=n$. Then we apply Kuo condensation in the same way as illustrated by Figures \ref{Newnegative}(c) and (d). This finishes our proof.
\end{proof}

\begin{proof} [Proof of Theorem \ref{genthm2}]
Theorem \ref{genthm2} can be proved by following the lines in the proof of Theorem \ref{genthm}. We also prove (\ref{geneq2}) by induction on $k+s+t$ with the base case $s=1$ following from Kenyon and Wilson's Theorem 5 in \cite{KW}.

For the induction step we assume (\ref{geneq2}) for all $\overline{\mathcal{Q}}$-type regions whose sum of $k$-, $s$-, and $t$- parameter is strictly less than $k+s+t$. To finish the proof we show that the minor $\overline{\SM}_{a,b}(k_1,\dotsc,k_s;t_1,\dotsc,t_{s-1})$ and the tiling polynomial $\Pol(\overline{\mathcal{Q}}_{x,h}(k_1,\dots,k_s;t_1,\dots,t_{s-1}))$ satisfy the same recurrence.

First by applying the Jaw Move (Lemma \ref{Jaw}) to the transpose matrix of an $(n+1)\times n$ matrix $N$, we get the following variation of Lemma \ref{Jaw}:
\begin{lem}\label{Jaw2}
Let $N$ be an $(n+1)\times n$ matrix. Then
\begin{equation}
\det N^{\widehat{e'}}\det N^{\widehat{d',f'}}_{\widehat{g'}}=\det N^{\widehat{d'}} \det N^{\widehat{e',f'}}_{\widehat{g'}}+\det N^{\widehat{f'}}\det N^{\widehat{d',e'}}_{\widehat{g'}},
\end{equation}
where the indices $g',f',e',d'$ appear in \emph{clockwise} order around the circle.
\end{lem}

Similar to recurrence (\ref{genminoreq1}) for the $\SM$-minors,  we get the following  recurrence for the $\overline{\SM}$-minors, by applying Lemma \ref{Jaw2} to the matrix $N:=M_{A\cup\{a+k_s\}}^{B}$ with the four indices $d'=a, e'=a+k_s,f'=a+k+t-1, g'=b+k-1$ (see Figure \ref{SMKuo}(b)):
\begin{align}\label{genminoreq2}
\overline{\SM}_{a,b}&(k_1,\dotsc,k_s;t_1,\dotsc,t_{s-1}) \overline{\SM}_{a+1,b}(k_1-1,\dotsc,k_s;t_1,\dotsc,t_{s-1}-1)
=\notag\\&\overline{\SM}_{a+1,b}(k_1,\dotsc,k_s;t_1,\dotsc,t_{s-1}-1)\overline{\SM}_{a,b}(k_1-1,\dotsc,k_s;t_1,\dotsc,t_{s-1})\notag\\
&+\overline{\SM}_{a,b}(k_1-1,\dotsc,k_s+1;t_1,\dotsc,t_{s-1}-1)\overline{\SM}_{a+1,b}(k_1,\dotsc,k_s-1;t_1,\dotsc,t_{s-1}).
\end{align}
Recurrence (\ref{genminoreq2}) is essentially the same as recurrence (\ref{genminoreq1}) in the case of the$\SM$-minors, the only difference here is that the $b$-parameter is now unchanged, as none of the four indices $d',e',f',g'$ is equal to $b$.

Next, we would like to show the same recurrence for the tiling polynomials of $\overline{\mathcal{Q}}$-type regions by applying Kuo condensation. A key of the application of Kuo condensation is finding suitable vertices $u,v,w,s$ as in Theorems \ref{Kuo1}, \ref{Kuo2}, and \ref{Kuo3}. However, as any $\overline{\mathcal{Q}}$-type region is obtained from a $\mathcal{Q}$-type region by a vertical reflection and a horizontal translation, we can apply the same Kuo's condensation theorem (i.e. Theorem \ref{Kuo1}, \ref{Kuo2}, or \ref{Kuo3}) with the four vertices $u,v,w,s$ obtained  by reflecting that in the case of $\mathcal{Q}$-type regions over a vertical line. For example, in  the case $1\leq t \leq h-k$, one can pick the four vertices $u,v,w,s$ on the region $\overline{\mathcal{Q}}_{x,h}(k_1,\dots,k_s-1;t_1,\dots,t_{s-1})$ as in Figure \ref{kuoQ2}, for  $s=2, k_1=3,k_2=3,t_1=2, h=11$. We notice that the four unit squares corresponding to  $u,v,w,s$ are now in the \emph{clockwise order} around the boundary of the region (as opposed to being in  the counter-clockwise order as in the case of the $\mathcal{Q}$-type region in the proof of Theorem \ref{genthm}).

\begin{figure}\centering
\includegraphics[width=8cm]{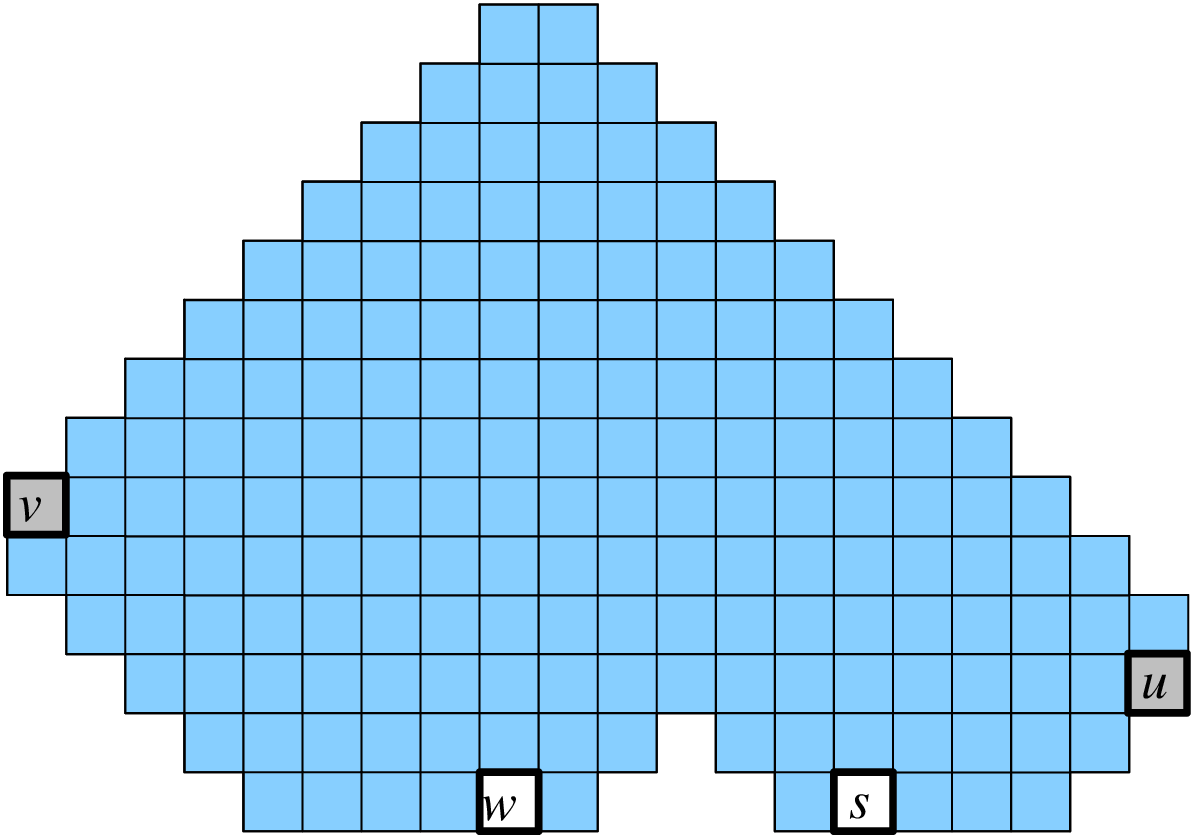}
\caption{How to apply Kuo condensation to a $\overline{\mathcal{Q}}$-type region, in the case $1\leq t \leq h-k$.}\label{kuoQ2}
\end{figure}

\begin{figure}\centering
\includegraphics[width=13cm]{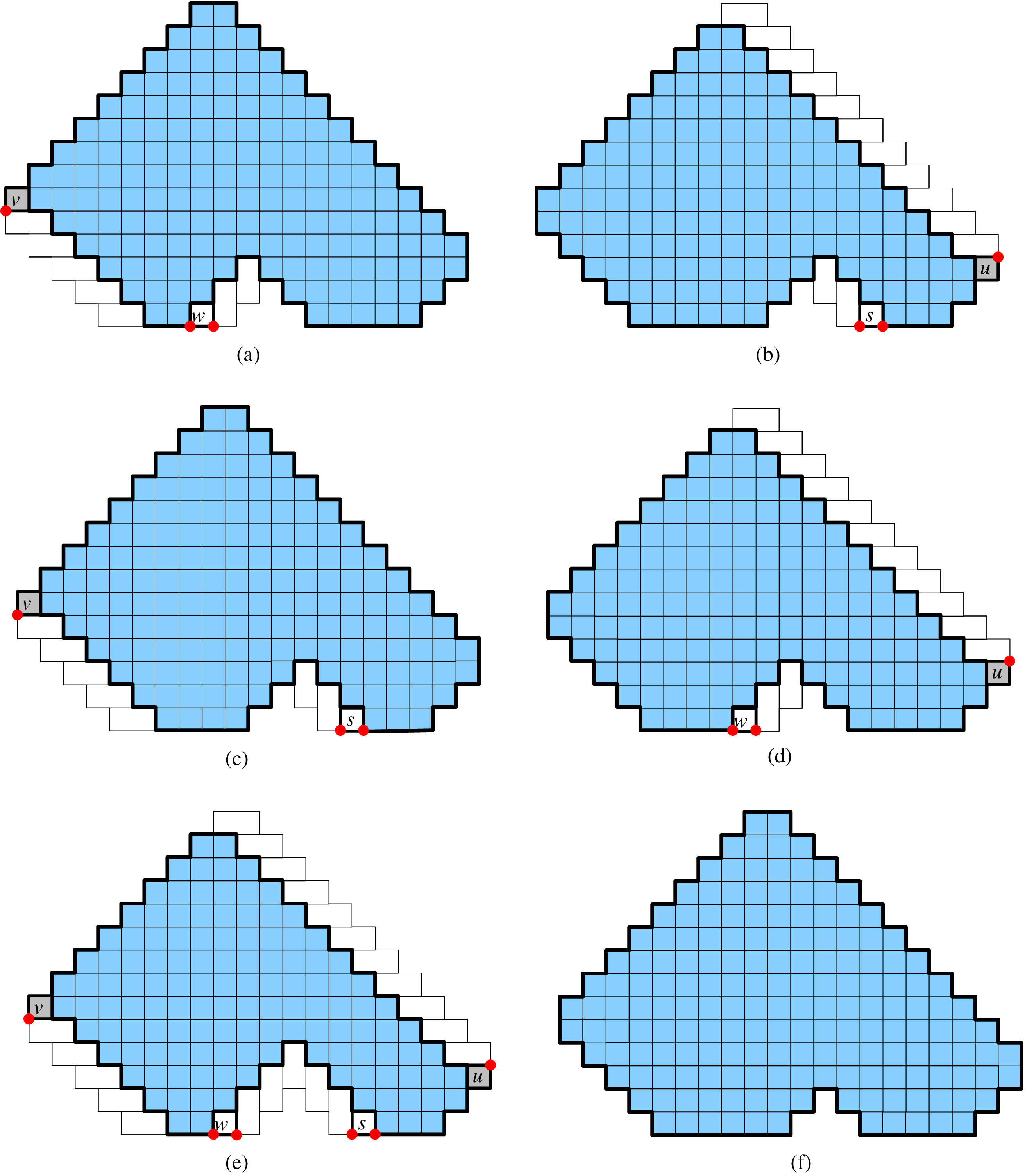}
\caption{Obtaining a recurrence for the tiling polynomials of $\overline{\mathcal{Q}}$-type regions, in the case $1\leq t \leq h-k$.}\label{kuoQ3}
\end{figure}

Similar to the proof of Theorem \ref{genthm}, by working on forced dominoes as in Figure \ref{kuoQ3},  we get a recurrence similar to recurrence (\ref{gentilingeq1}) in the proof of Theorem \ref{genthm}:
\begin{align}\label{gentilingeq2}
\Pol(\overline{\mathcal{Q}}_{x,h}&(k_1,\dots,k_s;t_1,\dots,t_{s-1}))\Pol(\overline{\mathcal{Q}}_{x-1,h}(k_1-1,\dots,k_s;t_1,\dots,t_{s-1}-1))=\notag\\
&\Pol(\overline{\mathcal{Q}}_{x,h}(k_1,\dots,k_s;t_1,\dots,t_{s-1}-1)) \Pol(\overline{\mathcal{Q}}_{x-1,h}(k_1-1,\dots,k_s;t_1,\dots,t_{s-1}))\notag\\
&+ \Pol(\overline{\mathcal{Q}}_{x-1,h}(k_1-1,\dots,k_s+1;t_1,\dots,t_{s-1}-1))\Pol(\overline{\mathcal{Q}}_{x,h}(k_1,\dots,k_s-1;t_1,\dots,t_{s-1})).
\end{align}
One readily sees that recurrence (\ref{gentilingeq2}) is essentially the same as recurrence (\ref{gentilingeq1}) in the case of the $\mathcal{Q}$-type regions. The only difference is that the $x$-parameters in the second, the fourth and the fifth terms are now reduced by $1$ (as opposed to being increased by $1$ in recurrence (\ref{gentilingeq1})).

This means that the minor $\overline{\SM}_{a,b}(k_1,\dotsc,k_s;t_1,\dotsc,t_{s-1})$ and the tiling polynomial $\Pol(\overline{\mathcal{Q}}_{x,h}(k_1,\dots,k_s;t_1,\dots,t_{s-1}))$ also satisfy the same recurrence when $1\leq t \leq h-k$.  By applying Kuo condensation similarly to the proof of Theorem \ref{genthm} (with reflected vertices $u,v,w,s$), we can get the same recurrence (\ref{gentilingeq2}) for the case $h\geq 0^+$ and $ t > h-k$ and for the case $h\leq 0^-$. This finishes our proof.
\end{proof}

\section{Open question for general circular minors}

We would like to know if the statement of Kenyon-Wilson Conjecture \ref{KWconj} holds for all circular minors whose index sets are both non-contiguous. Our data shows that there are several classes of such minors that can be written as the tiling polynomial of a region similar to our regions $\mathcal{Q}_{x,h}(k_1,\dots,k_s;t_1,\dots,t_{s-1})$ and $\overline{\mathcal{Q}}_{x,h}(k_1,\dots,k_s;t_1,\dots,t_{s-1})$.

\subsection*{Acknowledgements}
I would like to thank  Pavlo Pylyavskyy for drawing my attention to Kenyon and Wilson's conjecture, and for fruitful discussions. I also thank David Wilson, Gregg Musiker, Igor Pak, Patricia Hersh, and Thomas Lam for many helpful comments.

\end{document}